\documentclass[12pt]{article}
\usepackage{amsmath, amsfonts, graphicx, amsthm, color, enumerate, hyperref, amssymb}

\allowdisplaybreaks

\usepackage{lineno}

\DeclareMathOperator{\id}{id}

\DeclareMathOperator{\Mat}{Mat}

\DeclareMathOperator{\G}{G}

\DeclareMathOperator{\q}{q}

\DeclareMathOperator{\spa}{span}

\DeclareMathOperator{\GL}{GL}

\DeclareMathOperator{\Sp}{Sp}

\DeclareMathOperator{\Orth}{O}

\DeclareMathOperator{\rad}{rad}

\newtheorem{thm}[subsection]{Theorem}
\newtheorem*{thm*}{Theorem}
\newtheorem{prop}[subsection]{Proposition}
\newtheorem*{prop*}{Proposition}
\newtheorem{lemma}[subsection]{Lemma}
\newtheorem*{lemma*}{Lemma}
\newtheorem{corollary}[subsection]{Corollary}
\newtheorem*{corollary*}{Corollary}
\newtheorem{rmk}[subsection]{Remark}

\newtheorem{definition}[subsection]{Definition}

\title{Fixed point groups of involutions of type $\Orth(\q,k)$ over a field of characteristic two}
\author{Mark Hunnell \\ \textit{Winston-Salem State University} \\ hunnellm@wssu.edu \\ \\
John Hutchens \\ \textit{Winston-Salem State University} \\ hutchensjd@wssu.edu }

\begin{document}
\maketitle

\vfill
\noindent Declarations of interest: none \\
This research did not receive any specific grant from funding agencies in the public, commercial, or not-for-profit sectors.
\newpage

\begin{abstract}
    For $\Orth(\q,k)$, the orthogonal group over a field $k$ of characteristic 2 with respect to a quadratic form $\q$, we discuss the isomorphism classes of fixed points of involutions.  When the quadratic space is either totally singular or nonsingular, a full classification of the isomorphism classes is given.  We also give some implications of these results for a general quadratic space over a field of characteristic $2$.  
\end{abstract}

\section{Introduction}

We continue the discussion from \cite{hhs21} on $\Orth(\q,k)$-conjugacy classes of involutions of $\Orth(\q,k)$, where $k$ is a field of characteristic 2, and turn our focus to the fixed point groups of such involutions. The goal of this research is to describe generalized symmetric spaces of the form $G(k)/H(k)$ where $G$ is an algebraic group defined over a field $k$, $H$ is the fixed point group of some automorphism of order $2$ on $G$ and $G(k)$ (respectively $H(k)$) denotes the $k$-rational points of $G$ (resp. $H$).   In particular we want to extend Helminck's  study of $k$-involutions and symmetric $k$-varieties \cite{he00} to include fields of characteristic $2$.  This has been studied for groups of type $\G_2$ and $\mathrm{A}_n$ in \cite{hs18, sc18} and over fields of characteristic not $2$ in \cite{do06, bhw15, he02, bdhw16, hu14, hu15,hu16}.  We also extend the results of Aschbacher and Seitz \cite{as76} who studied similar structures for finite fields of characteristic $2$.    Symmetric spaces were first studied by Gantmacher in \cite{ga39} in order to classify simple real Lie groups.  This was continued in \cite{be57} by Berger who provides a complete classification of symmetric spaces for simple real Lie algebras.   

We refer to \cite{hl04} for notation and vocabulary concerning quadratic forms over fields of characteristic $2$.  Orthogonal and symplectic groups over fields of characteristic $2$ have been studied extensively.   In \cite{ch66} Cheng Hao discusses automorphisms of the orthogonal group over perfect fields of characteristic $2$ when the quadratic form is nondefective.  Pollak discusses orthogonal groups over global fields of characteristic $2$ in the case the quadratic form is nondefective in \cite{po70} and Connors writes about automorphism groups of orthogonal groups over fields of characteristic $2$ in \cite{co73,co74,co75,co76} for a nondegenerate quadratic form.  In \cite{wi78} Wiitala uses the Jordan canonical form to study involutions of the orthogonal group over fields of characteristic 2 in the nondefective case. In \cite{di51} Jean Dieudonn\'e discusses algebraic homogeneous spaces over fields of characteristic $2$.

In the following we establish isomorphism classes for the fixed point groups of involutions of orthogonal groups for quadratic spaces with a totally singular quadratic form and a non-singular quadratic form.  We also provide a characterization of the fixed point group for a general quadratic form defining an orthogonal group over a field of characteristic 2.  When possible, this is done without assuming the form is nondegenerate.

Over fields of characteristic not $2$ the $G(k)$-conjugacy classes of automorphisms of order $2$ are in bijection with $G(k)$-conjugacy classes of fixed point groups of involutions on $G$ as seen in \cite{hw93}. However, in characteristic $2$ there are examples of involutions that are not $G(k)$-conjugate with isomorphic fixed point groups, see Theorem \ref{orth_diag_fixed_pts}, as well as Theorem 5.27 in \cite{hs18}.

In the Preliminaries section we recall definitions, notation, and useful results concerning quadratic forms and bilinear forms over fields of characteristic $2$ and associated algebraic groups.  This section also includes notation from the theory of symmetric $k$-varieties.  Next we recall results concerning involutions of $\Orth(\q,k)$ when $k$ is a field of characteristic $2$ in Section $3$.  This is followed by describing isomorphism classes of fixed point groups for different types of vector spaces.  In Section $4$ we consider the fixed point group for involutions of $\Orth(\q,k)$ where $\q$ is the quadratic form of a totally singular vector space.  Then in Section $5$ we discuss the fixed point group of involutions where $\q$ is the quadratic form of a nonsingular vector space. Finally, Section 6 has an overview of the implications for a general vector space over a field of characteristic $2$.

\section{Preliminaries}
The following definitions can be found in \cite{hl04}.  We also refer to \cite{gr02} often, which uses some of the vocabulary for orthogonal groups differently than \cite{hl04}.  When this happens, we default to the definitions of \cite{hl04}.  

Let $k$ be a field of characteristic 2 and $W$ a vector space defined over $k$.  We call $\q:W \rightarrow k$ a \emph{quadratic form} if it satisfies $\q(aw) = a^2 \q(w)$ for all $a \in k$, $w \in W$. There exists a corresponding symmetric bilinear form $\mathrm{B}: W \times W \rightarrow k$ such that $\q(w+w') = \q(w)+\q(w') + \mathrm{B}(w,w')$ for all $w,w' \in W$.  Over fields of characteristic $2$ nonsingular symmetric bilinear forms are also symplectic.

The pair $(W,\q)$ is called a \emph{quadratic space}. The subspace of $W$ 
\[ \mathrm{rad}(W) = \{ x \in W \ | \ \mathrm{B}(x,w)= 0,\text{ for all } w \in W\}. \]
is called the \emph{radical} of $W$.

Given a quadratic form, there exists a basis of $W$, consisting of $x_i, y_i, g_j$, where $i \in \{1, 2, \hdots, r\}$ and $j \in \{1, 2, \hdots, s\}$ and field elements $a_i, b_i, c_j \in k$ such that 
\[\q(w) = \displaystyle \sum_{i=1}^r (a_i e_i^2 + e_if_i + b_if_i^2) + \sum_{j=1}^s c_j z_j^2\]
when $w = \sum_{i=1}^r (e_ix_i + f_iy_i) + \sum_{j=1}^s z_j g_j$.  We denote this quadratic form by
\begin{equation}
\label{quadratic_signature}
\q = [a_1,b_1] \perp [a_2,b_2] \perp \cdots \perp [a_r,b_r]\perp \langle c_1, c_2, \hdots, c_s  \rangle 
\end{equation}
and so $\rad(W)= \spa \{ g_1, g_2, \hdots, g_s\} $.  The representation of $\q$ given in \ref{quadratic_signature} is called the \emph{quadratic signature} of $\q$. We say that such a quadratic form is of type $(r,s)$.  A nonzero vector $w\in W$ is an \emph{isotropic vector} if $\q(w)=0$, $V$ is an \emph{isotropic vector space} if it contains isotropic elements and \emph{anisotropic} otherwise.  The space $W$ is called \emph{nonsingular} if $s=0$, and \emph{totally singular} if $r=0$. Furthermore, $W$ is called \emph{nondefective} if $s=0$ or $\rad(W)$ is anisotropic.  A \emph{hyperbolic plane} has a quadratic form isometric to the form $[0,0]$ and will be denoted by $\mathbb{H}$.  We will call $\q'$ a \emph{subform} of $\q$ if there exists a form $\mathrm{p}$ such that $\q \cong \q' \perp \mathrm{p}$. In this case, $p$ is called a completion of $\q'$ with respect to $\q$. 

Sometimes we will need to refer to a diagonal bilinear form
\[ \mathrm{B}(w,w') = a_1w_1w'_1 + a_2w_2w'_2 + \cdots + a_lw_lw'_l, \]
which we denote $\langle a_1, a_2, \cdots, a_l \rangle_{\mathrm{B}}$ following \cite{hl04}.  Throughout the paper we will use the calligraphy math font to denote sets of vectors in $W$ such as $\mathcal{U} = \{u_1, u_2, \ldots, u_l\}$ and $\overline{\mathcal{U}} = \spa \mathcal{U}$.   The next result is Proposition 2.4 from \cite{hl04}. 

\begin{prop}\label{hofflaghdecomp}
Let $\q$ be a quadratic form over $k$.  Then
\[ \q \cong m \times \mathbb{H} \perp \widetilde{\q_r} \perp \widetilde{\q_s} \perp d \times \langle 0 \rangle, \]
with $\widetilde{q_r}$ nonsingular, $\widetilde{\q_s}$ totally singular and $\widetilde{\q_r} \perp \widetilde{q_s}$ anisotropic.  The form $\widetilde{\q_r} \perp \widetilde{\q_s}$ is uniquely determined up to isometry.  In particular $m$ and $d$ are uniquely determined.  
\end{prop}

We call $m$ the \emph{Witt index} and $d$ the \emph{defect} of $\q$.  If 
\[ \q \cong m \times \mathbb{H} \perp \widetilde{\q_r} \perp  d \times \langle 0 \rangle \perp \widetilde{\q_s}, \]
with respect to the basis 
\[ \{ x_1,y_1, \ldots x_m,y_m, \ldots , x_r,y_r, g_1, \ldots g_d, g_{d+1}, \ldots , g_{s} \}, \]
we will call 
\[ \mathrm{def}(W) = \mathrm{span}\{g_1,\ldots g_d \}, \]
the \emph{defect} of $W$.  If $\mathcal{U}$ is a basis for a subspace $\overline{\mathcal{U}}$ of $W$, we will refer to the restriction of $\q$ to $\overline{\mathcal{U}}$ with respect to this basis by $\q_{\mathcal{U}}$.

Let $G$ be an algebraic group, then an automorphism $\theta: G \rightarrow G$ is an \emph{involution} if $\theta^2 = \id$, $\theta \neq \id$. The automorphism $\theta$ is a $k$-\emph{involution} if in addition $\theta(G(k)) = G(k)$, where $G(k)$ denotes the $k$-rational points of $G$.  We define the {\em fixed point group} of $\theta$ in $G(k)$ by
\[ G(k)^{\theta} = \{ \gamma \in G(k) \ | \ \theta \gamma \theta^{-1} = \gamma \}, \]
this is often denoted $H(k)$ or $H_k$ in the literature when there is no ambiguity with respect to $\theta$. Let $\mathcal{I}_{\tau}$ denote conjugation by $\tau$.

We often consider groups that leave a bilinear form or a quadratic form invariant.  If $\mathrm{B}$ is a bilinear form on a nonsingular vector space $W$ we will denote the {\em symplectic group} of $(W,\q)$ over a field $k$ by
\[ \Sp(\mathrm{B},k) = \{ \varphi \in \GL(W) \ | \ \mathrm{B}(\varphi(w), \varphi(w')) = \mathrm{B}(w,w') \text{ for } w,w' \in W \}. \]
The classification of involutions for $\Sp(\mathrm{B},k)$ for a field $k$ such that $\mathrm{char}(k) \neq 2$ has been studied in \cite{bhw15}.  For any quadratic space $W$ over a field $k$ we will denote the {\em orthogonal group} of $(W,\q)$ by
\[ \Orth(\q,k) = \{  \varphi \in \GL(W) \ | \ \q(\varphi(w)) = \q(w) \text{ for all } w \in W \}. \]
When $W$ is nonsingluar we have $\Orth(\q,k) \subset \Sp(\mathrm{B},k)$ if $\mathrm{B}$ is the bilinear form associated with $\q$.

We will need to make use of some facts about quadratic spaces stated in the following lemmas.  The first statement outlines some standard isometries for quadratic forms over a field of characteristic $2$, and the second allows us to express $\q$ using a different completion of the nonsingular space.  These results and more like them appear in \cite{hl04}.

\begin{lemma}
Let $\q$ be a quadratic form on a vector space $W$ and suppose $\alpha \in k$.  Then the following are equivalent representations of $q$ on $W$:
\begin{enumerate}
\item $[a,a'] = [a,a+a'+1] = [a',a] = [ \alpha^2 a, \alpha^{-2} a'] $
\item $[a,a'] \perp [b,b'] = [a+b,a'] \perp [b,a'+b'] = [b,b'] \perp [a,a']$
 \end{enumerate}
\end{lemma}

\begin{lemma} Let $a_i, b_i, c_i \in k$ for $1 \leq i \leq n$.  Suppose $\{a_1, ... , a_n \}$ and $\{b_1, ... , b_n \}$ span the same vector space over $k^2$ and $\q = [a_1,c_1] \perp \hdots \perp [a_n , c_n]$.  Then there exist $c_i^{\prime} \in k$, $1 \leq i \leq n$, such that $\q = [b_1,c_1^{\prime}] \perp \hdots \perp [b_n , c_n^{\prime}]$.
\end{lemma}

For a map $\theta \in \GL(W)$ we define the \emph{residual space} of $\theta$ to be the image of $\theta + \id$ and the \emph{residue} of $\theta$ to be the dimension of the residual space.  

\section{Summary of involution results}

Here we summarize the results from \cite{hhs21} that we need to talk about the classification of the fixed points groups of each $\Orth(\q,k)$-conjugacy class of an involution.   

If $W$ is a nonsingular vector space with respect to $\q$ we can define a \emph{transvection} on $W$ induced by $w \in W$ and $a \in k$
\[ \tau_{w,a}(z) = z + a \mathrm{B}(w,z)w, \]
and an \emph{orthogonal transvection} is of the form $\tau_{w, \frac{1}{\q(w)}}$ for $w \in W$ with $\q(w) \neq 0$.  In this case we denote the transvection by $\tau_w$.

If $\mathcal{U}$ is the set of mutually orthogonal vectors inducing a product of orthogonal transvections $\tau$ we call $\mathcal{U}$ the \emph{inducing set} for 
\[ \tau = \tau_{\mathcal{U}} = \tau_{u_l} \cdots \tau_{u_2}\tau_{u_1}. \]
Let $\mathcal{U} = \{u_1, \ldots, u_l\}$ and $\mathcal{X} = \{ x_1, \ldots, x_l\}$ be sets of mutually orthogonal vectors in a nonsingular subspace of $W$.  In what follows we will factor involutions into a product of involutions of a given type (orthogonal transvection, basic null, basic radical), such a factorization is called \emph{reduced} if it uses the minimum number of factors.  We refer to the number of factors in a reduced factorization as the \emph{length} of an involution.

\begin{lemma}\label{equal_tr}
Two orthogonal involutions given by reduced products of orthogonal transvections are equal, $\tau_{u_l} \cdots \tau_{u_2} \tau_{u_1} = \tau_{x_l} \cdots \tau_{x_2} \tau_{x_1}$, if and only if $\overline{\mathcal{U}} = \overline{\mathcal{X}}$.
and
\[ \left\langle \q(u_1), \q(u_2), \ldots, \q(u_l) \right\rangle_{\mathrm{B}} \cong  \left\langle \q(x_1), \q(x_2), \ldots, \q(x_l) \right\rangle_{\mathrm{B}}. \]
\end{lemma}

The classification of $\Orth(\q,k)$-conjugacy classes of involutions given by a product of transvections in a nonsingular space can then be described.
\begin{thm} \label{conj_transv}
Let $\tau_{u_l} \cdots \tau_{u_2} \tau_{u_1}$ and $\tau_{x_l} \cdots \tau_{x_2} \tau_{x_1}$ be orthogonal involution written as a product of transvections on $W$ such that $\phi \in \Orth(\q,k)$.  Then 
\[ \phi \tau_{u_l} \cdots \tau_{u_2} \tau_{u_1}\phi^{-1} = \tau_{x_l} \cdots \tau_{x_2} \tau_{x_1} \]
if and only if 
\[ \left\langle \q(u_1), \q(u_2), \ldots, \q(u_l) \right\rangle_{\mathrm{B}} \cong  \left\langle \q(x_1), \q(x_2), \ldots, \q(x_l) \right\rangle_{\mathrm{B}}, \]
and $\dim(\overline{\mathcal{U}}) = \dim(\overline{\mathcal{X}})$.
\end{thm}

The last stipulation, $\dim(\overline{\mathcal{U}}) = \dim(\overline{\mathcal{X}})$,  distinguishes between elements $\tau$ whose reduced length is $\mathrm{res}(\tau) + 1$, called hyperbolic involutions, and those whose reduced length is $\mathrm{res}(\tau)$, called diagonal. 

There is one other type of nonsingular involution called a null involution.   
\begin{definition} A plane $P = \spa\{ x, y \}$ is hyperbolic (or Artinian) if both of the following are satisfied:
	\begin{enumerate}
		\item $\q(x) = \q(y) = 0$ 
		\item $B(x,y) \neq 0.$
	\end{enumerate}
\end{definition}

If $x, y$ span a hyperbolic plane, we can rescale to assume $B(x,y) =1$.  Proposition 188.2 of \cite{sn89} guarantees that every nonsingular nonzero isotropic vector is contained in a hyperbolic plane.  
\begin{definition}
Let $\eta$ be an involution of $\Orth(\q,k)$ where $(W,\q)$ is a quadratic space, and let $\mathbb{P}$ be the orthogonal sum of two hyperbolic planes.  Then an involution $\eta$ is called a basic null involution in $\mathbb{P}$ if all of the following are satisfied:
	\begin{enumerate}
		\item $\eta$ leaves $\mathbb{P}$ invariant
		\item $\eta$ fixes a 2-dimensional subspace of $\mathbb{P}$ where every vector has norm zero
		\item $\eta \vert_{\mathbb{P}^C} = \id_{\mathbb{P}^C}$, where $\mathbb{P}^C$ is the complement of $\mathbb{P}$ in $W$.
	\end{enumerate}
\end{definition}

A \emph{null involution} is a product of basic null involutions. Note that if $\eta$ is a null involution and $\tau$ is a transvection, then $\tau \eta$ is a product of transvections \cite{sn89}, so we only need to consider the conjugacy of null involutions themselves.

\begin{prop} \label{conj_null}
Two null involutions are $\Orth(\q,k)$-conjugate if and only if they have the same length.
\end{prop}

Finally we consider involutions acting only on the radical of $W$.  

\begin{prop} \label{Oqk_radV}
$\Orth(\q_{\rad(W)},k) \cong \GL_d(k) \ltimes \Mat_{d,s-d}(k)$ where $d$ is the defect of $\rad(W)$.
\end{prop}
\begin{proof}
Let $d$ be the defect of $W$.  We can choose a basis of the form
\[ \{g_1, \ldots, g_d, g_{d+1}, \ldots, g_s\} \]
where $\q(g_i)=0$ for $1\leq i \leq d$ 
\[ \mathcal{G}_{\mathrm{is}} = \{g_{1}, \ldots, g_d\}, \]
and the subspace spanned by
\[  \mathcal{G}_{\mathrm{an}} = \{g_{d+1}, \ldots, g_s\} \]
is anisotropic.  We claim we can represent an isometry on $\mathrm{rad}(W)$ by 
\[ \Psi= \begin{bmatrix}
	\psi & M \\
	0 & \id
	\end{bmatrix}. \]
We must have $\Psi|_{\mathcal{G}_{\mathrm{an}}} = \id$ otherwise there would exist two elements in $\mathcal{G}_{\mathrm{an}}$ with the same norm, and since the space is totally singular the sum of the two vectors would have norm zero. The restriction $\Psi|_{\mathcal{G}_{\mathrm{is}}}= \psi \in \GL_d(k)$ since $\psi$ must be invertible but has no other contraints due to the bilinear form and the norm being identically zero for $\overline{\mathcal{G}_{\mathrm{is}}}$.  The isometric image of an element from $\overline{\mathcal{G}_{\mathrm{is}}}$ cannot include any components from $\overline{\mathcal{G}_{\mathrm{an}}}$ since that would imply that $\overline{\mathcal{G}_{\mathrm{an}}}$ was not anisotropic.  Lastly the $M$ block is free since we can add any isotropic vector from the the radical to any vector in $W$ and leave the quadratic form (and bilinear form) invariant. 

We can factor 
\[ \Psi= \begin{bmatrix}
	\psi & M \\
	0 & \id
	\end{bmatrix}=
	\begin{bmatrix}
	\psi & 0 \\
	0 & \id
	\end{bmatrix}
	\begin{bmatrix}
	\id & M \\
	0 & \id
	\end{bmatrix} \]
where $\bigl[\begin{smallmatrix} \psi & M \\ 0 & \id\end{smallmatrix} \bigr]\bigl[\begin{smallmatrix} \id & B \\ 0 & \id\end{smallmatrix} \bigr]\bigl[\begin{smallmatrix} \psi & M \\ 0 & \id\end{smallmatrix} \bigr]^{-1} = \bigl[\begin{smallmatrix} \id & \psi B \\ 0 & \id\end{smallmatrix} \bigr]$.
\end{proof}

A \emph{radical involution} is an element $\rho \in \Orth(\q,k)$ or order $2$ defined on vectors in the radical of $W$.  Notice that a nontrivial orthogonal transformation on $\rad(W)$ detects a defective vector in $W$, since if $\rho(g) = g'$ then $\q(g+g') = \q(g) + \q(g') = 0$.  A \emph{basic radical involution} is a radical involution $\rho_i \in \Orth(\q,k)$ such that $\rho_i(g_i) = g_i'$, where $g_i, g_i'$ are linearly independent vectors in $\rad(W)$ with $\q(g_i) = \q(g_i')$, and $\rho_i$ acts trivially on the complement of $\overline{\{g_i, g_i'\}}$.  All radical involutions can be written as a finite product of basic radical involutions.

\begin{prop}
Two basic radical involutions $\rho_1, \rho_2$ are $\Orth(\q,k)$-conjugate if and only if $\rho_1$ and $\rho_2$ act non-trivially on isometric vectors.
\end{prop}

\begin{corollary} \label{conj_rad}
All radical involutions of length $m$ with quadratic signature isometric to
\[  \langle \q(g_1), \q(g_2), \ldots, \q(g_m) \rangle, \]
are $\Orth(\q,k)$-conjugate.
\end{corollary}

\section{Totally singular spaces}

We are now ready to discuss the isomorphism classes of the fixed point groups of our involutions of $\Orth(\q,k)$, which can be used to study the corresponding symmetric $k$-varieties.  First, we discuss the fixed point groups of involutions acting on a totally singular vector space.  This in turn leads to a discussion about fixed point groups of diagonal involutions on nonsingular subspaces.

Let us consider a vector space $W$ such that $W= \rad(W)$.   Recall that in such a vector space any quadratic form will be of the form
\[ \q = \langle \q(g_1), \q(g_2), \ldots, \q(g_s) \rangle, \]
and by Proposition \ref{hofflaghdecomp} we have an isometry 
\[ \q \cong \langle 0, \ldots, 0, \q(g_1), \q(g_2), \ldots, \q(g_{s-d}) \rangle, \]
where $d$ is the defect of $\q$.

By Proposition \ref{Oqk_radV} we have $\Orth(\q_{\rad(W)}, k) \cong \GL_d(k) \ltimes \Mat_{d,s-d}(k)$.  For an element $\delta \in \Orth(\q,k)$ to be fixed by conjugation with the involution $\rho$ we need $\rho \delta \rho = \delta$.

\begin{prop}
If $\q$ is a quadratic form on $W = \mathrm{rad}(W)$ and $\rho \in \Orth(\q,k)$ is an involution on $W$ such that $\rho(g_i) = g_i'$  and quadratic signature
\[  \langle \q(g_1),\q(g_2), \cdots, \q(g_n) \rangle \]
then $\Orth(\q,k)^{\mathcal{I}_{\rho}} \cong \Orth(\q_{\mathcal{G} \cup \mathcal{H}}, k) \ltimes \Mat_{n,s-n}(k)$
where 
\begin{align*}
\mathcal{G}' &=  \{g_1',\ldots, g_n' \} \\
\mathcal{G} &=  \{g_1,\ldots, g_n \} \\
\mathcal{H} &=  \{h_1, \ldots, h_{s-2n}\}
\end{align*}
 with $\rho(h_i) = h_i$ and $\overline{ \mathcal{G}' \cup \mathcal{G} \cup \mathcal{H} } = W$.
\end{prop}
\begin{proof}
Let $\langle \q(g_1), \q(g_2), \ldots, \q(g_n) \rangle$ be the signature of a radical involution $\rho$ such that $\rho(g_i) = g_i'$ and $g_{i}'\neq g_i$ for any $i$ or $i'$.  Notice that $\rho(g_i + g_i') = g_i + g_i'$ and further that $\rho(g_i) = g_i + (g_i + g_i')$.  Now we order a basis for $W$, 
\[ \mathcal{W} = \{ g_1+g_1', g_2+g_2', \ldots, g_n+g_n', h_1, h_2, \ldots, h_{s-2n}, g_1,g_2, \ldots, g_n \} \]
such that $\overline{ \mathcal{W} } = W$ and $\rho(h_i) = h_i$.  With respect to this basis we can represent $\rho$ as the block matrix
\[ \rho = \begin{bmatrix}
	\id & 0 & \id \\
	0 & \id & 0 \\
	0 & 0 & \id
	\end{bmatrix}, \]
and so any $\delta \in \Orth(\q, k)^{\mathcal{I}_{\rho}}$
\[  \delta = \begin{bmatrix}
	\delta_{11} & \delta_{12} & \delta_{13} \\
	\delta_{21} & \delta_{22} & \delta_{23} \\
	\delta_{31} & \delta_{32} & \delta_{33} \\
	\end{bmatrix}, \]
such that $\delta\rho = \rho \delta$ is of the form
\[  \delta = \begin{bmatrix}
	\delta_{11} & \delta_{12} & \delta_{13} \\
	0 & \delta_{22} & \delta_{23} \\
	0 & 0 & \delta_{11} \\
	\end{bmatrix}= 
	\begin{bmatrix}
	\delta_{11} & 0 & 0 \\
	0 & \delta_{22} & \delta_{23} \\
	0 & 0 & \delta_{11} \\
	\end{bmatrix}
	\begin{bmatrix}
	\id & \delta_{12}' & \delta_{13}' \\
	0 & \id & 0 \\
	0 & 0 & \id \\
	\end{bmatrix} \]	
where $\delta_{ij}' = \delta_{11}^{-1}\delta_{ij}$.  Notice that $\delta_{12}'$ and $\delta_{13}'$ can take any matrix value, since $\overline{ \{ g_1+g_1', \ldots, g_n+g_n'\}}$ is contained in the defect of $W$ and the span of the remaining basis vectors would be 
\[ \mathcal{H} \cup \mathcal{G}  =  \{ h_1, \dots, h_{s-2n}, g_1, \ldots, g_n \}, \]
Now the subgroup
\[ \mathsf{M}=\left\{ \begin{bmatrix}
	\delta_{11} & 0 & 0 \\
	0 & \delta_{22} & \delta_{23} \\
	0 & 0 & \delta_{11} \\
	\end{bmatrix} \bigg| \bigl[\begin{smallmatrix} \delta_{22} & \delta_{23} \\ 0 & \delta_{11} \end{smallmatrix} \bigr] \in \Orth(\q_{\mathcal{H} \cup \mathcal{G} },k, \overline{\mathcal{H}} ) \right\}, \]
is such that $\mathsf{M} \cong \Orth(\q_{\mathcal{H} \cup \mathcal{G} },k, \overline{\mathcal{H}} )$ which is the subgroup of $\Orth(\q_{\mathcal{H} \cup \mathcal{G} },k)$ that leaves the subspace $\overline{\mathcal{H}}$ invariant.
The subgroup 
\[ \mathsf{M}'= \left\{ \begin{bmatrix}
	\id  & \delta_{12}' & \delta_{13}' \\
	0 & \id & 0 \\
	0 & 0 & \id \\
	\end{bmatrix} \bigg|  \delta_{12}' \in \Mat_{n, d'}(k), \delta_{13}' \in \Mat_{n,s-n-d'}(k) \right\}, \]
is such that $\mathsf{M}' \cong \Mat_{n,s-n}(k)$ and $\mathsf{M}$ is a normal subgroup of $\Orth(\q,k)^{\mathcal{I}_{\rho}}$,
so we have that $\Orth(\q,k)^{\mathcal{I}_{\rho}} \cong \Orth(\q_{\mathcal{H} \cup \mathcal{G} },k, \overline{\mathcal{H}} ) \ltimes \Mat_{n,s-n}(k)$.
\end{proof}

\begin{rmk}
Notice that if $W = \mathrm{rad}(W)$ is anisotropic then $\Orth(\q,k) = \{id\}$ so there are no involutions, and on the other hand if $W$ is totally isotropic we have that $\Orth(\q,k) \cong \GL_s(k)$.
\end{rmk}

\section{Nonsingular spaces}

Now we look at fixed point groups of diagonal involutions on nonsingular spaces. 

\begin{prop} \label{tranv_fixU}
Let $\tau = \tau_{u_l} \cdots \tau_{u_2}\tau_{u_1}$ be a reduced diagonal orthogonal involution with inducing set $\mathcal{U}$, then $\varphi \in \Orth(\q,k)^{\mathcal{I}_{\tau}}$ if and only if $\varphi( \overline{ \mathcal{U}})=\overline{\mathcal{U}}$ . 
\end{prop}
\begin{proof}
Let $\varphi \in \Orth(\q,k)^{\mathcal{I}_{\tau}}$ notice that $\tau \varphi \tau^{-1} = \varphi$ if and only if $\varphi \tau \varphi^{-1} = \tau$ and so
\[ \varphi  \tau_{u_l} \cdots \tau_{u_2}\tau_{u_1} \varphi^{-1} =  \tau_{\varphi(u_l)} \cdots \tau_{\varphi(u_2)}\tau_{\varphi(u_1)} = \tau_{u_l} \cdots \tau_{u_2}\tau_{u_1}. \]
By Lemma \ref{equal_tr} $\varphi(\overline{ \mathcal{U} }) = \overline{ \mathcal{U}}$.

Now assume $\varphi \in \Orth(\q,k)$ and $\varphi( \overline{ \mathcal{U}}) = \overline{ \mathcal{U} }$ then 
\[ \langle \q( \varphi(u_1) ), \q( \varphi(u_2)), \ldots, \q( \varphi(u_l)) \rangle_B = \langle \q(u_1), \q( u_2), \ldots, \q( u_l) \rangle_B. \]
So by Theorem \ref{conj_transv} we have $\varphi \tau \varphi^{-1} = \tau$.
\end{proof}

We begin with the case that the length of $\tau$ is maximized in the nonsingular space. We denote by $\varphi_{\mathcal{U}}$ the matrix representation of the resriction of $\varphi$ to the subspace $\overline{ \mathcal{U}}$ with respect to the basis $\mathcal{U}$.  Recall that 
\[ \q_{\mathcal{U}} \cong \langle \q(u_1), \q(u_2), \ldots, \q(u_l) \rangle, \]
since $\overline{ \mathcal{U}}$ is a totally singular subspace of $W$.  In this case, Proposition \ref{hofflaghdecomp} yields

\[ \q_{\mathcal{U}} \cong m' \times \langle 0 \rangle \perp \langle c_{m'+1}, \ldots, c_{l} \rangle. \] 
This corresponds to a basis of $\mathcal{U}' = \{u_1', u_2', \ldots, u_l' \}$ such that $\q(u_i^{\prime})=c_i$ for $1 \leq i \leq l$ and $c_i=0$ for $1\leq i \leq m'$.  Since the vectors in $\overline{\mathcal{U}}$ are not in $\rad(W)$ there exist non-singular completions within $W$ that form a basis for a nonsingular subspace.  A nonsingular completion of $\q$ is of the form
\[ [0,\q(v_1)] \perp [ 0, \q(v_2) ] \perp \ldots \perp [0, \q(v_{m'}) ] \perp [ c_{m'+1}, \q(v_{m'+1})] \perp \ldots \perp [c_l, \q(v_l)], \]
where $\mathrm{B}(u_i',v_i)=1$, $\mathrm{B}(u_i',v_j)=0$ when $j\neq i$, and $\mathcal{V} = \{v_1, \ldots, v_l\}$ which spans a totally singular subspace.  Notice that for $1\leq i \leq m'$ we can choose $v_i' = v_i + \q(v_i)u_i'$.  We will refer to the vectors of this completion as $\mathcal{V}'$.  Now
\[ \mathrm{B}(u_i', v_i') = \mathrm{B}(u_i', v_i) + \q(v_i)\mathrm{B}(u_i', u_i') = \mathrm{B}(u_i', v_i) = 1, \]
and
\[ \q(v_i') = \q(v_i) + \q(v_i)^2\q(u_i')+\q(v_i)B(v_i,u_i') = \q(v_i) + \q(v_i) = 0. \]
So there is a basis $\mathcal{W}$ of this nonsingular space of the form
\[\mathcal{W}'=\mathcal{U}'\cup \mathcal{V}'= \{u_1',v_1', u_2', v_2', \ldots, u_l',v_l' \}, \]
with the corresponding quadratic form
\[ \q_{\mathcal{W}} \cong m' \times [0,0] \perp [c_{m'+1}, d_{m'+1}] \perp \ldots \perp [c_l,d_l], \]
where $\q(v_i') = d_i$ for $1 \leq i \leq l$. 

Let us consider $\varphi \in \Orth(\q,k)^{\mathcal{I}_{\tau}}$.  Since $\varphi$ leaves $\overline{ \mathcal{U}}$ invariant we have 
\[ \varphi=\begin{bmatrix}
			\varphi_{\mathcal{U}} & \varphi_{\mathcal{U}}A \\
			0 & \varphi_{\mathcal{V}}
			\end{bmatrix}, \]
where  $\varphi_{\mathcal{V}} :\overline{ \mathcal{V}} \to \overline{ \mathcal{V}}$ and $A: \overline{ \mathcal{V}} \to \overline{ \mathcal{U}}$.  Considering the bilinear form we have
\[ \mathrm{B}(\varphi(u_i), \varphi(v_j) ) = \mathrm{B}(\varphi_{\mathcal{U}}(u_i), \varphi_{\mathcal{V}}(v_j) + \varphi_{\mathcal{U}}(Av_j))= \mathrm{B}(\varphi_{\mathcal{U}}(u_i), \varphi_{\mathcal{V}}(v_j) ) = \mathrm{B}(u_i,v_j). \]
So, we have shown the following
\begin{prop} \label{phiUstar}
The map $\varphi_{\mathcal{V}}$ is completely determined by $\varphi_{\mathcal{U}}$.
\end{prop}
We will now refer to $\varphi_{\mathcal{V}} = \varphi_{\mathcal{U}}^*$ and using the matrix representation of $\varphi_{\mathcal{U}}$ we have $\varphi_{\mathcal{U}}^*=(\varphi_{\mathcal{U}}^{-1})^T$.  We can represent $\varphi \in \Orth(\q,k)^{\mathcal{I}_{\tau}}$ with the block matrix
\[ \varphi= \begin{bmatrix}
			\varphi_{\mathcal{U}} & \varphi_{\mathcal{U}}A \\
			0 & \varphi_{\mathcal{U}}^*
			\end{bmatrix}  . \]

We define the additive group
\[ \mathsf{A}(\q_{\mathcal{U}},k) = \{ A \in \text{Lin}( \overline{ \mathcal{V}},\overline{ \mathcal{U}}) \ | \ \q(Av) = \mathrm{B}(v,Av) \}. \]
Notice that $\mathsf{A}(\q_{\mathcal{U}},k)$ embeds into $\Orth(\q,k)^{\mathcal{I}_{\tau}}$ as an additive group since
\[ \q( (A+C)v) =  \q(Av) + \q(Cv) = \mathrm{B}(v,Av) + \mathrm{B}(v,Cv) = \mathrm{B}(v,(A+C)v), \]
for $A,C \in \mathsf{A}(\q_{\mathcal{U}},k)$.  We will write $\varphi^{-*}$ for $(\varphi^{-1})^*$.  The following results give us some insight into the structure of the fixed point groups.

\begin{lemma} \label{adjoint}
If $u \in \overline{ \mathcal{U}}$ and $v \in \overline{ \mathcal{V}}$, then $\mathrm{B}(\varphi_{\mathcal{U}}(u),v) = \mathrm{B}(u,\varphi_{\mathcal{U}}^{-*}(v))$. 
\end{lemma}
\begin{proof}
\begin{align*}
\mathrm{B}(u,\varphi_{\mathcal{U}}^{-*}(v)) &= \mathrm{B}(\varphi (u), \varphi \varphi_{\mathcal{U}}^{-*}(v)) \\
&= \mathrm{B}(\varphi_{\mathcal{U}}(u), (\varphi_{\mathcal{U}}^* + \varphi_{\mathcal{U}}A)\varphi_{\mathcal{U}}^{-*}(v))\\
&=\mathrm{B}(\varphi_{\mathcal{U}}(u), v + \varphi_{\mathcal{U}}A\varphi_{\mathcal{U}}^{-*}(v) ) \\
&=\mathrm{B}(\varphi_{\mathcal{U}}(u), v),
\end{align*}
since $\varphi_{\mathcal{U}}A\varphi_{\mathcal{U}}^{-*}(v) \in \mathcal{U}$.
\end{proof}

In the next result we see that this gives us an isomorphism of the additive group $\mathsf{A}(\q_{\mathcal{U}},k)$.

\begin{lemma} 
For $\varphi_{\mathcal{U}} \in \Orth(\q_\mathcal{U},k)$ and $A \in \mathsf{A}(\q_{\mathcal{U}},k)$ we have that 
\[ \Phi_{\varphi_{\mathcal{U}}}(A) = \varphi_{\mathcal{U}} A \varphi_{\mathcal{U}}^{-*}, \]
is an isomorphism of $\mathsf{A}(\q_{\mathcal{U}},k)$.
\end{lemma}
\begin{proof}
The map $\Phi_{\varphi_{\mathcal{U}}}(A)$ is a homomorphism with respect to addition in $\mathsf{A}(\q_{\mathcal{U}},k)$ and
\begin{align*}
\q(\varphi_{\mathcal{U}} A \varphi_{\mathcal{U}}^{-*}(v)) &= \q(A \varphi_{\mathcal{U}}^{-*}(v)) \\
&= B(\varphi_{\mathcal{U}}^{-*}(v),A \varphi_{\mathcal{U}}^{-*}(v)) \\
&= B(v,\varphi_{\mathcal{U}} A \varphi_{\mathcal{U}}^{-*}(v)),
\end{align*}
by Lemma \ref{adjoint} and the fact that $\varphi_{\mathcal{U}} \in \Orth(\q_{\mathcal{U}},k)$, we have the result.
\end{proof}

\begin{prop}\label{tranv_fix_rel}
For $\varphi= \bigl[\begin{smallmatrix} \varphi_{\mathcal{U}} & \varphi_\mathcal{U}A \\ 0 & \varphi_{\mathcal{U}}^*\end{smallmatrix} \bigr]$ we have $\varphi \in \Orth(\q,k)^{\mathcal{I}_{\tau}}$ if and only if $\varphi_{\mathcal{U}} \in \Orth(\q,k)$ and $\q(v+Av) = \q(\varphi_{\mathcal{U}}^*(v))$ for $v \in \overline{ \mathcal{V}}$ a nonsingular completion of $\overline{\mathcal{U}}$ in $W$.
\end{prop}
\begin{proof}
This can be shown with the following computation.  Let us compute
\begin{align*}
\q(\varphi(v)) &= \q(\varphi_{\mathcal{U}}^*(v) + \varphi_{\mathcal{U}}(Av) ) \\
\q(v) &= \q(\varphi_{\mathcal{U}}^*(v)) +\q( \varphi_{\mathcal{U}}(Av) ) + \mathrm{B}(\varphi_{\mathcal{U}}^*(v), \varphi_{\mathcal{U}}(Av) ) \\
\q(v) &= \q(\varphi_{\mathcal{U}}^*(v)) +\q( Av ) + \mathrm{B}(v, Av ) \\
\q(v) +\q( Av ) + \mathrm{B}(v, Av ) &= \q(\varphi_{\mathcal{U}}^*(v))  \\
\q(v+Av) &= \q(\varphi_{\mathcal{U}}^*(v)).
\end{align*}
\end{proof}

\begin{rmk}
We note here that it is not necessarily the case that for every $\varphi_{\mathcal{U}} \in \Orth(\q_{\mathcal{U}},k)$ that there exists an element of the form $\bigl[\begin{smallmatrix} \varphi_{\mathcal{U}} & 0 \\ 0 & \varphi_{\mathcal{U}}^*\end{smallmatrix} \bigr] \in \Orth(\q,k)^{\mathcal{I}_{\tau}}$.  It is true, however, that for every $\varphi_{\mathcal{U}} \in \Orth(\q_{\mathcal{U}},k)$ there exists $A_{\varphi}: \overline{ \mathcal{V}} \to \overline{ \mathcal{U}}$ such that  $\bigl[\begin{smallmatrix} \varphi_{\mathcal{U}} & \varphi_\mathcal{U}A_{\varphi} \\ 0 & \varphi_{\mathcal{U}}^*\end{smallmatrix} \bigr] \in \Orth(\q,k)^{\mathcal{I}_{\tau}} \subset \Orth(\q,k)$, which we state next.
\end{rmk}

\begin{lemma} \label{tranv_fix_exist}
If $\tau= \tau_{u_l} \cdots \tau_{u_2}\tau_{u_1}$ is a diagonal involution on $W$ of length $l$, $\dim (W)=2l$, and $\varphi_{\mathcal{U}} \in \Orth(\q_{\mathcal{U}},k)$, then there exists $\varphi=\bigl[\begin{smallmatrix} \varphi_{\mathcal{U}} & \varphi_\mathcal{U}A \\ 0 & \varphi_{\mathcal{U}}^*\end{smallmatrix} \bigr] \in \Orth(\q_{\mathcal{U}},k)^{\mathcal{I}_{\tau}}$.
\end{lemma}

\begin{proof}
We know that elements $\varphi \in \Orth(\q,k)^{\mathcal{I}_{\tau}}$ leave $\overline{ \mathcal{U} }$ invariant and must be of the form
\[ \varphi= \begin{bmatrix}
			\varphi_{\mathcal{U}} & A \\
			0 & \varphi_{\mathcal{U}}^*
			\end{bmatrix}. \]
We now consider $\varphi_{\mathcal{U'}}$ where $\mathcal{U}'$ is the basis of $\overline{\mathcal{U}}$ with respect to 
\[ \q'=m\times [0,0] \perp  [\q(u'_{m+1}),\q(v'_{m+1})] \perp \ldots \perp [\q(u'_l),\q(v'_l) ] \]
and $[\q(u'_{m+1}),\q(v'_{m+1})] \perp \ldots \perp [\q(u'_l),\q(v'_l) ]$ is anisotropic, so that the Witt index of $\q$ is $m$.  Then we can write 
\[ \varphi= \begin{bmatrix}
			\varphi_{\mathcal{U'}} & A' \\
			0 & \varphi_{\mathcal{U'}}^*
			\end{bmatrix}, \]
where
\[ \varphi_{\mathcal{U'}} = \begin{bmatrix}
			P_1 & P_2 \\
			0 & \id
			\end{bmatrix} \text{ which give us } 
			\varphi_{\mathcal{U'}}^* = \begin{bmatrix}
			P_1^* & 0 \\
			P_2^TP_1^* & \id
			\end{bmatrix}. \]
Notice that we can choose $P_1 \in \GL_m(k)$ and $P_2 \in \Mat_{m,l-m}(k)$ since $P_1$ is acting on $\overline{ \{u_1',\ldots,u_m'\}}$ all of the vectors are orthogonal to one another and all have norm zero.  The map $P_2$ adds norm zero vectors to the anisotropic vectors, and again all vectors in the space $\mathcal{U}'$ are mutually orthogonal, so there are no further constraints.

Let us define $\mathcal{U'}_{\mathbb{H}} =  \{u_1',\ldots, u_m'\}$ and $\mathcal{V'}_{\mathbb{H}} =  \{v_1',\ldots, v_m'\}$, the defective subspaces of $\mathcal{U}$ and $\mathcal{V}$, where $\overline{\mathcal{U}_{\mathbb{H}}' \cup \mathcal{V}_{\mathbb{H}}'}$ forms the hyperbolic subspace of $W = \overline{ \mathcal{W}}$.

We need to establish that $A'$ exists where
\[ A' = \begin{bmatrix}
			A_1' & A_2' \\
			A_3' & A_4'
			\end{bmatrix}. \]
We can choose $A_2'=A_3'=A_4'=0$ and $A_1': \overline{\mathcal{V'}_{\mathbb{H}} } \to  \overline{ \mathcal{V'}_{\mathbb{H}}}$ such that $A_1' v'_j = \q(P_2^TP_1^*v'_j)P_1 u'_j$. First notice that for $v_j'$
\begin{align*}
\q(\varphi(v_j')) &= \q(A_1'v_j' + P_1^* v_j' + P_2^TP_1^* v_j') \\
&= \q(A_1'v_j' ) + \q(P_1^*v_j' + P_2^TP_1^* v_j') + B(A_1'v_j', P_1^* v_j' + P_2^TP_1^* v_j') \\
&= \q(P_1^*v_j') + \q( P_2^TP_1^* v_j') + B(A_1'v_j', P_1^* v_j' ) \\
&= \q( P_2^TP_1^* v_j') + B(A_1'v_j', P_1^* v_j' ) \\
&= \q( P_2^TP_1^* v_j') + B(\q(P_2^TP_1^*v'_j)P_1 u'_j, P_1^* v_j' ) \\
&=0,
\end{align*}
since $A_1'v_j', P_1^* v_j' \in \mathcal{V}_{\mathbb{H}}$ and $B(u_j',v_j') = 1$.  We use Proposition \ref{tranv_fix_rel} noting that $\varphi=\left[ \begin{smallmatrix}  \varphi_{\mathcal{U}} & A \\ 0 & \varphi_{\mathcal{U}}^* \end{smallmatrix} \right]=\left[ \begin{smallmatrix}  \varphi_{\mathcal{U}} & \varphi_{\mathcal{U}} \varphi_{\mathcal{U}}^{-1} A \\ 0 & \varphi_{\mathcal{U}}^* \end{smallmatrix} \right]$.  Now if $v_j' \in   \mathcal{V'}_{\mathbb{H}}$
\begin{align*}
\q(v_j' + \varphi_{\mathcal{U}}^{-1}(Av_j')) &= \q(v_j') + \q(\varphi_{\mathcal{U}}^{-1}(Av_j')) + \mathrm{B}(v_j', \varphi_{\mathcal{U}}^{-1}(Av_j')) \\
&= 0 + \q(\q(P_2 P_1^* v_j')u_j') + \mathrm{B}(v_j',\q(P_2 P_1^* v_j')u_j') \\
&= 0 + \q(P_2 P_1^* v_j')\mathrm{B}(v_j',\q(P_2 P_1^* v_j')u_j') \\
&=\q(P_2 P_1^* v_j') \\
&=\q(\varphi_{\mathcal{U}}^*(v_j')).
\end{align*}
If $v_j$ is a basis vector not in $\overline{ \mathcal{V'}_{\mathbb{H}}}$ then then $Av_j'=0$ and Proposition \ref{tranv_fix_rel} is trivially satisfied.
\end{proof}

We define $A_{\varphi} \in \Mat_{l,l}(k)$ as the diagonal matrix with $\q(\varphi_{\mathcal{U}}^* (v_j'))$ in the $j$-th diagonal entry $1\leq j \leq m$ and $0$ elsewhere, i.e. $A_{\varphi}v_j = \q(\varphi_{\mathcal{U}}^* (v_j'))u_j$ for $1\leq j \leq m$ and $A_{\varphi}v_j' = 0$ for $m+1 \leq j \leq l$.

\begin{prop} \label{AforP}
$A_{\varphi \theta} =  \theta_{\mathcal{U}}^{-1} A_{\varphi} \theta_{\mathcal{U}}^* + A_{\theta}$
\end{prop}
\begin{proof}
If we let 
\[ \theta_{\mathcal{U}} = \begin{bmatrix}
					Q_1 & Q_2 \\
					0 & \id
					\end{bmatrix},\]
then we have
\[\varphi = \begin{bmatrix}
			P_1 & P_2 &  P_1 A_{\varphi 1} & 0 \\
			0 & \id & 0 & 0  \\
			0 & 0 & P_1^* & 0 \\
			0 & 0 & P_2^TP_1^* & \id 
			\end{bmatrix} \text{ and }
 \ \theta = \begin{bmatrix}
			Q_1 & Q_2 &  Q_1 A_{\theta 1} & 0 \\
			0 & \id & 0 & 0  \\
			0 & 0 & Q_1^* & 0 \\
			0 & 0 & Q_2^T Q_1^* & \id 
			\end{bmatrix}.
\]
This gives us the product
\[\varphi \theta= \begin{bmatrix}
			P_1Q_1 & P_1Q_2 + P_2 &  P_1 A_{\varphi 1}Q_1 + P_1Q_1A_{\theta 1}  & 0 \\
			0 & \id & 0 & 0  \\
			0 & 0 & P_1^*Q_1^* & 0 \\
			0 & 0 & (P_2^TP_1^*+Q_2^T)Q_1^* & \id 
			\end{bmatrix} \]
and computing $(\theta_{\mathcal{U}}^{-1} A_{\varphi} \theta_{\mathcal{U}}^* + A_{\theta})v_j' $ for $1\leq j \leq m$
\begin{align*}
(\theta_{\mathcal{U}}^{-1} A_{\varphi} \theta_{\mathcal{U}}^* + A_{\theta})v_j' &= P_1Q_1(Q_1^{-1} A_{\varphi} Q_1^* + A_{\theta})v_j'   \\
&=Q_1^{-1} A_{\varphi} Q_1^*v_j' + A_{\theta}v_j'  \\
&= Q_1^{-1} \q(P_2^T P_1^* Q_1^* v_j) Q_1 u_j' + \q(Q_2^T Q_1^* v_j) u_j'  \\
&=  \q(P_2^T P_1^* Q_1^* v_j)u_j' + \q(Q_2^T Q_1^* v_j) u_j' \\
&= \q( (P_2^T P_1^* + Q_2^T)Q^*v_j ) u_j' \\
&= A_{\varphi \theta} v_j'.
\end{align*}
Notice that $(\theta_{\mathcal{U}}^{-1} A_{\varphi} \theta_{\mathcal{U}}^* + A_{\theta})v_j' =0$ for $m \leq j \leq l$.
\end{proof}

We make the identification
\[ \begin{bmatrix}
\varphi_{\mathcal{U}} & \varphi_{\mathcal{U}}A_{\varphi} \\
0 & \varphi_{\mathcal{U}}^*
\end{bmatrix} \mapsto (\varphi_{\mathcal{U}}, 0), \]
which gives us the product 
\begin{equation} \label{product_nonsing}(\varphi_{\mathcal{U}}, A)(\theta_{\mathcal{U}}, C) = (\varphi_{\mathcal{U}} \theta_{\mathcal{U}}, \theta_{\mathcal{U}}^{-1} A \theta_{\mathcal{U}}^* + C)
\end{equation}
and
\begin{equation}\label{inv_nonsing}
(\varphi_{\mathcal{U}}, A)^{-1} = (\varphi_{\mathcal{U}}^{-1}, \varphi_{\mathcal{U}} A \varphi_{\mathcal{U}}^{-*}).
\end{equation}

\begin{corollary}
$(\varphi_{\mathcal{U}}, 0 ) ( \theta_{\mathcal{U}},0) = ( \varphi_{\mathcal{U}}\theta_{\mathcal{U}}, 0)$
\end{corollary}

\begin{lemma} \label{tranv_fix_form}
If $\tau$ is a diagonal involution on $W$ of length $l$ and $\dim (W)=2l$, then $\varphi \in \Orth(\q,k)^{\mathcal{I}_{\tau}}$ can be written as $(\varphi_{\mathcal{U}},A)$ where $\varphi_{\mathcal{U}} \in \Orth(\q_{\mathcal{U}},k)$ and $A \in \mathsf{A}(\q_{\mathcal{U}},k)$.
\end{lemma}
\begin{proof}
By Proposition \ref{phiUstar} we have determined that elements in $\Orth(\q,k)^{\mathcal{I}_{\tau}}$ are of the form 
\[ \varphi= \begin{bmatrix}
\varphi_{\mathcal{U}} & \varphi_{\mathcal{U}}M \\
0 & \varphi_{\mathcal{U}}^*
\end{bmatrix} \]
and we have made the identification 
\[ \begin{bmatrix}
\varphi_{\mathcal{U}} & \varphi_{\mathcal{U}}A_{\varphi} \\
0 & \varphi_{\mathcal{U}}^*
\end{bmatrix} \mapsto (\varphi_{\mathcal{U}}, 0). \]
This gives us by (\ref{product_nonsing}) that $(\varphi,A) = (\varphi,0)(\id,A)$ such that $A \in \mathsf{A}(\q_{\mathcal{U}},k)$ and we can set $M=A_{\varphi}+A$.  For $v \in \overline{\mathcal{V}}$, a nonsingular completion of $\overline{\mathcal{U}}$, we check
\begin{align*}
\q(v + (A_{\varphi} + A)v) &= \q(v) + \q(A_{\varphi}v + Av) + \mathrm{B}(v, A_{\varphi}v) + \mathrm{B}(v,Av) \\
&= \q(v) + \q(A_{\varphi}v) + \q(Av) + \mathrm{B}(v, A_{\varphi}v) + \mathrm{B}(v,Av)\\
&= \q(v) + \q(A_{\varphi}v) + \mathrm{B}(v, A_{\varphi}v)  \\
&=\q(v+A_{\varphi}v) \\
&= \q(\varphi^*_{\mathcal{U}}(v)) \\
\end{align*}
and so, by Proposition \ref{tranv_fix_rel} we have the result.
\end{proof}

Finally we show that all elements of the fixed point group can be decomposed as the semi-direct product of the two subgroups of $\Orth(\q,k)^{\mathcal{I}_{\tau}}$ described above. 

\begin{thm}
When $\tau = \tau_{u_l} \cdots \tau_{u_2}\tau_{u_1}$ is a diagonal involution of length $l$ on a nonsingular space of dimension $2l$
\[ \Orth(\q,k)^{\mathcal{I}_{\tau}} \cong \Orth(\q_{\mathcal{U}},k) \ltimes \mathsf{A}(\q_{\mathcal{U}},k). \]
\end{thm}
\begin{proof}
By Proposition \ref{tranv_fixU} we know that $\varphi$ must leave $\overline{ \mathcal{U}}$ invariant.   By Proposition \ref{phiUstar} we know that the action of $\varphi$ leaving $\overline{\mathcal{V}}$ invariant is completely determined by $\varphi_{\mathcal{U}}$.  Lemma \ref{tranv_fix_exist} tells us that for any $\varphi_\mathcal{U} \in \Orth(\q_{\mathcal{U}}, k)$ there exists $(\varphi_{\mathcal{U}}, 0) \in \Orth(\q,k)^{\mathcal{I}_{\tau}}$.   Now consider any $(\varphi_{\mathcal{U}}, C) \in \Orth(\q,k)^{\mathcal{I}_{\tau}}$ then we can factor
\[ (\varphi_{\mathcal{U}},C) = (\varphi_{\mathcal{U}},0)( \id, C)\]
and we know $(\varphi_{\mathcal{U}}, 0), (\varphi_{\mathcal{U}},C) \in \Orth(\q,k)$ so we must have $( \id, C) \in \mathsf{A}(\q_{\mathcal{U}},k)$.

Now $\mathsf{A}(\q_{\mathcal{U}},k)$ is a normal subgroup of $\Orth(\q,k)^{\mathcal{I}_{\tau}}$, since 
\begin{align*}
(\theta_{\mathcal{U}},D)(\id, C)(\theta_{\mathcal{U}},D)^{-1} &= (\theta_{\mathcal{U}},D)(\id, C)(\theta_{\mathcal{U}}^{-1},\theta_{\mathcal{U}}D \theta_{\mathcal{U}}^{-*})  \\
&= (\theta_{\mathcal{U}}, D+C)(\theta_{\mathcal{U}}^{-1},\theta_{\mathcal{U}}D \theta_{\mathcal{U}}^{-*}) \\
&= (\theta_{\mathcal{U}}\theta_{\mathcal{U}}^{-1}, \theta_{\mathcal{U}}(D+C)\theta_{\mathcal{U}}^{-*} + \theta_{\mathcal{U}} D \theta_{\mathcal{U}}^{-*}) \\
&= (\id, \theta_{\mathcal{U}}C \theta_{\mathcal{U}}^{-*}).
\end{align*}

\end{proof}

Here we look at the fixed point group of a diagonal involution on a nonsingular vector space $W$ with $\dim(W) = 2r$ with $2l< 2r$.

\begin{prop}
\label{nonsing_diag_fixed_pts}
An element $\varphi \in \Orth(\q,k)^{\mathcal{I}_{\tau}}$ where $\tau = \tau_{\mathcal{U}}$ is a product of orthogonal transvections can be written in the form
\[ \begin{bmatrix}
	\varphi_{\mathcal{U}} & \varphi_{\mathcal{U}} C & \varphi_{\mathcal{U}} A \\
	0 & \varphi_{\mathcal{X}} & \varphi_{\mathcal{X}} D  \\
	0 & 0 &  \varphi_{\mathcal{U}}^*
	\end{bmatrix}, \]
 where $\varphi_{\mathcal{U}} \in \Orth(\q_{\mathcal{U}},k)$, $\varphi_{\mathcal{X}} \in \mathrm{Sp}(\mathrm{B}_{\mathcal{X}},k)$ with $\varphi_{\mathcal{U}}^*$ is the unique linear map such that $\mathrm{B}(\varphi_{\mathcal{U}}(u),\varphi_{\mathcal{U}}^*(v))=\mathrm{B}(u,v)$.
\end{prop}
\begin{proof}  
The map $\varphi$ acts on $W$ using ordered basis 
\[ \mathcal{W} = \{u_1, \ldots, u_l, x_1,\ldots, x_{r-l}, y_1, \ldots, y_{r-l}, v_1, \ldots, v_l \} \]
where $\mathcal{U}$ is the inducing set for $\tau$, $\mathcal{V}$ is a nonsingular completion of $\mathcal{U}$ and $\tau(x) = x$ for all $x \in \overline{\mathcal{X}}$ such that 
\[ \mathcal{X} = \{ x_1,\ldots, x_{r-l}, y_1, \ldots, y_{r-l} \}. \]
Notice that the blocks are zero under $\varphi_{\mathcal{U}}$ in the first column since by Proposition \ref{tranv_fixU} the subspace $\overline{ \mathcal{U}}$ is left invariant by any element in $\Orth(\q,k)^{\mathcal{I}_{\tau}}$ where
\[ \tau = \tau_{\mathcal{U}} = \tau_{u_l} \cdots \tau_{u_2} \tau_{u_1}. \]
In the second column the third block is zero since $\overline{\mathcal{U}}$ is left invariant by $\varphi$ and $\overline{\mathcal{X}} \subset \overline{\mathcal{U}}^{\perp}$.  
\end{proof}

We have established $\mathrm{B}(\varphi_{\mathcal{U}}(u), \varphi_{\mathcal{U}}^*(v) ) = \mathrm{B}(u,v)$, and we will see that $C$ and $D$ have a similar relationship.  

\begin{prop} \label{fix_prop_1}
Let $\mathcal{X}$ be a basis for $\left(\overline{ \mathcal{U} \cup \mathcal{V} } \right)^{\perp}$ such that $\tau(x) = x$ for all $x\in \overline{\mathcal{X}}$.  The map $\varphi_{\mathcal{X}} : \overline{ \mathcal{X}} \to \overline{ \mathcal{X}}$ leaves the bilinear form invariant.
\end{prop}
\begin{proof}
Let $x,x' \in \overline{ \mathcal{X}}$, then
\begin{align*}
\mathrm{B}(x,x') &= \mathrm{B}(\varphi(x), \varphi(x')) \\
&= \mathrm{B}(\varphi_{\mathcal{X}}(x) + \varphi_{\mathcal{U}}(Cx), \varphi_{\mathcal{X}}(x') + \varphi_{\mathcal{U}}(Cx')) \\
&= \mathrm{B}(\varphi_{\mathcal{X}}(x), \varphi_{\mathcal{X}}(x'))
\end{align*}
since $\overline{ \mathcal{U}} \subset \overline{ \mathcal{U}}^{\perp}$.
\end{proof}

The group $\Orth(\q,k)^{\mathcal{I}_{\tau}}$ has a subgroup isomorphic to $\Orth(\q_{\mathcal{U}},k)$ with elements of the form
\[ \begin{bmatrix}
	\varphi_{\mathcal{U}} & 0  & \varphi_{\mathcal{U}}A_{\varphi} \\
	0 & \id  & 0 \\
	0 & 0 & \varphi_{\mathcal{U}}^*
	\end{bmatrix}. \]
	
Next, we include Theorem 14.1 from \cite{gr02}, which is used several times.  The following results lead to a description of the fixed point groups of involutions on $\Orth(\q,k)$.

\begin{thm}
\label{uniqueC}
Let $(W,\q)$ be a quadratic space with maximal nonsingular subspace $W_1$ where $2r= \mathrm{dim}(W_1)$ and $s=\mathrm{dim}(\mathrm{rad}(W))$ with $\q|_{\mathrm{rad}(W)}$ anisotropic.  If $\sigma \in \Orth(W)$ then $\sigma|_{\mathrm{rad}{(W)}} = \id_{\mathrm{rad}(W)}$, $\sigma|_{W_1} = \sigma_0 + \sigma_1$, where $\sigma_0:W_1 \to \mathrm{rad}(W)$ and $\sigma_1 \in \Sp(W_1)$.  If $\tau \in \mathrm{Sp}(W_1)$ then there exists $\sigma \in \Orth(W)$ with $\sigma_1 = \tau$ if and only if $\q(\tau(w)) + \q(w) \in \q(\mathrm{rad}(W))$ for all $w \in W_1$, moreover if $\sigma$ exists it is unique.
\end{thm}

\begin{prop} \label{fix_prop_2}
$\mathrm{B}(v,Cx) = \mathrm{B}(Dv,x)$.
\end{prop}
\begin{proof}
Consider 
\begin{align*}
0 &= \mathrm{B}(v,x) \\
&= \mathrm{B}(\varphi(v), \varphi(x)) \\
&= \mathrm{B}(\varphi_{\mathcal{U}}(Av) + \varphi_{\mathcal{U}}^*(v) + \varphi_{\mathcal{X}}(Dv), \varphi_{\mathcal{U}}(Cx) + \varphi_{\mathcal{X}}(x)) \\
&= \mathrm{B}(\varphi_{\mathcal{X}}(Dv),\varphi_{\mathcal{X}}(x)) + \mathrm{B}(\varphi_{\mathcal{U}}^*(v), \varphi_{\mathcal{U}}(Cx)) \\
&= \mathrm{B}(Dv, x) + \mathrm{B}(v, Cx).
\end{align*}
\end{proof}
 Since $D$ is completely determined by $C$ we will denote $D$ by $C^{\dagger}$.   So, we can rewrite the result in Proposition \ref{fix_prop_2}
 
 \begin{equation}
 \mathrm{B}(v,Cx) = \mathrm{B}(C^{\dagger}v,x)
 \end{equation}
 
 In fact we can choose a basis for $W$ such that $\mathrm{B}_{\mathcal{U}\mathcal{V}}: \overline{\mathcal{U}} \times \overline{\mathcal{V}} \to k$ is represented by $\id$ and $\mathrm{B}_{\mathcal{X}}: \overline{\mathcal{X}} \times \overline{\mathcal{X}} \to k$ is represented by a diagonal block matrix of the form
 \[ J = \begin{bmatrix}
 		0 & \id \\
 		\id & 0
 		\end{bmatrix} , \]
which gives us the equation $C = C^{\dagger T}J$ or in other words $C^\dagger = JC^T $.  From this we can see that $(C + C')^{\dagger} = C^{\dagger} + C'^{\dagger}$
 
We choose the basis
\[ \mathcal{W} =  \{ u_1,u_2, \ldots, u_l, x_1, x_2, \ldots, x_{r-l}, y_1, y_2, \ldots, y_{r-l
},v_1, v_2, \ldots, v_l \} \]
and write $\varphi \in \Orth(\q,k)^{\mathcal{I}_{\tau}}$ as
\[ \varphi =  \begin{bmatrix}
	\varphi_{\mathcal{U}} & \varphi_{\mathcal{U}} C & \varphi_{\mathcal{U}} A \\
	0 &  \varphi_{\mathcal{X}} & \varphi_{\mathcal{X}} C^{\dagger} \\
	0 & 0 & \varphi_{\mathcal{U}}^*  
	\end{bmatrix}. \]
The image of $x \in \overline{ \mathcal{X}}$ under $\varphi \in \Orth(\q,k)^{\mathcal{I}_{\tau}}$ gives us the following relations
\begin{align*}
\q(x) &= \q(\varphi(x)) \\
&= \q(\varphi_{\mathcal{U}}(Cx) + \varphi_{\mathcal{X}}(x)) \\
&= \q(\varphi_{\mathcal{U}}(Cx)) + \q(\varphi_{\mathcal{X}}(x)) \\
&= \q(Cx) + \q(\varphi_{\mathcal{X}}(x)).
\end{align*}
So further we see that
\[ \q(x+Cx) = \q(\varphi_{\mathcal{X}}(x)). \]

From here we begin building an argument that for every pair $(\varphi_{\mathcal{X}}, C)$ such that $\q(x+Cx) = \q(\varphi_{\mathcal{X}}(x))$, we have a linear transformation $M$ such that the triple $(\varphi_{\mathcal{U}}^*, C^{\dagger}, M)$ leaves $\q$ invariant on the image of $\overline{\mathcal{V}}$ while preserving $\mathrm{B}$ on $\varphi(W)$.

\begin{lemma} \label{tranv_Sp_M}
Let  $\q_{\mathcal{U}}$ be anisotropic and $T_{x,\omega} \in \GL(W)$ be of the form
\[  T_{x,\omega}=	\begin{bmatrix}
	\id_{\mathcal{U}} &  C & M \\
	0 &  \tau_{x,\omega} & \tau_{x,\omega} C^{\dagger} \\
	0 & 0 & \id_{\mathcal{V}}
	\end{bmatrix} \]
with $\tau_{x,\omega} \in \Sp(\mathrm{B}_{\mathcal{X}},k)$ and $\q(\tau_{x,\omega} (z) + Cz) = \q(z)$ for all $z\in \mathcal{X}$.  Then there exists an $M:\overline{\mathcal{V}} \to \overline{\mathcal{U}}$ such that $T_{x,\omega} \in \Orth(\q,k)$.
\end{lemma}

\begin{proof}
let $x \in \overline{ \mathcal{X}}$. To show that we have such an $M$ for the pair $(\tau_{x,\omega}, C)$, consider 
\[ 	\begin{bmatrix}
	\id &  C & M \\
	0 &  \tau_{x,\omega} & \tau_{x,\omega} C^{\dagger} \\
	0 & 0 & \id
	\end{bmatrix} \]
We can assume that $x=x_1$ is a basis vector such that there exists $y_1 \in \overline{\mathcal{X}}$ with $\mathrm{B}(x_1,y_1)=1$.

Since $C: \overline{\mathcal{X}} \to \overline{\mathcal{U}}$ and $\overline{\mathcal{U}}$ is a totally singular anisotropic subspace of $W$ we know that the image of $\q_{\mathcal{U}}$ is $k^2[\q(u_1),\q(u_2),\ldots,\q(u_l) ]$, i.e. the $k^2$ vector space spanned by the norms of the elements of $\mathcal{U}$.

Assume $\q(z) = \q( \tau_{x,\omega}(z) + Cz)$, then we have $\q(z) +  \q( \tau_{x,\omega}(z)) = \q(Cz)$ with $\q(Cz) \in k^2[\q(u_1),\q(u_2),\ldots,\q(u_l) ]$.

We extend $\{x_1,y_1\}$ to a basis $\{x_1,y_1, \dots, x_n,y_n\}$ for $\mathcal{X}$ such that
\[ \q_{\mathcal{X}} = [\q(x_1),\q(y_1)] \perp \cdots \perp [\q(x_n),\q(y_n)]. \]
We can now be more explicit about the image of $C$. If $\q(z) = \q(\tau_{x_1,\omega}(z))$, then $\q(Cz)=0$ and thus $Cz=0$ since $\overline{\mathcal{U}}$ is anisotropic.  Now suppose $\q(z) \neq \q(\tau_{x_1,\omega}(z))$ and let $c_{jj'}$ be the coefficient of $u_j$ of $Cx_{j'}$ and $d_{jj'}$ be the coefficient of $u_j$ for $Cy_{j'}$.  Since $\tau_{x_1,\omega}(x_{j'}) = x_{j'}$, $\tau_{x_1,\omega}(y_{j'}) = y_{j'}$ for $j'\neq 1$ and $\tau_{x_1,\omega}(y_{1})= y_1 + \omega x_1$ we have
\begin{align}
\displaystyle\sum_{j=1}^l c_{jj'}^2\q(u_j) &= 0 \text{ for all } j,j' \label{czero} \\
\displaystyle\sum_{j=1}^l d_{jj'}^2\q(u_j) &= 0 \text{ for all } j \text{ and all } j' \neq 1 \label{dzero} \\
\displaystyle\sum_{j=1}^l d_{j1}^2\q(u_j) &= \q(y_1) + \q(y_1 + \omega x_1) \label{dnotzero}
\end{align}
By (\ref{czero}) we have $c_{jj'}=0$ for all $j$ and $j'$ and by (\ref{dzero}) $d_{jj'}=0$ for all $j' \neq 1$.  Solving for $\q(x_1)$ in (\ref{dnotzero}) we have
\[ \q(x_1) = \dfrac{1}{\omega} + \displaystyle\sum_{j=1}^l \left( \dfrac{d_{j1}}{\omega} \right)^2 \q(u_j). \]

Now we consider the image of $v_i$  
\[ T_{x_1,\omega}(v_i) = v_i + \tau_{x_1,\omega}(C^{\dagger}v_i) + Mv_i, \]
and computing the norm of both sides of this equation we end up with
\begin{align*}
\q(v_i) &= \q(v_i) + \q(Mv_i) + \mathrm{B}(v_i,Mv_i) + \q(\tau_{x_1,\omega}(C^{\dagger}v_i) )\\
\q(Mv_i) + \mathrm{B}(v_i,Mv_i) &= \q(\tau_{x_1,\omega}(C^{\dagger}v_i) ) \\
\q(Mv_i) + \mathrm{B}(v_i,Mv_i) &= \q( C^{\dagger}v_i + \omega \mathrm{B}(x_1,C^{\dagger}v_i)x_1 ) \\
\q(Mv_i) + \mathrm{B}(v_i,Mv_i) &= \q( C^{\dagger}v_i)+ \q(\omega \mathrm{B}(x_1,C^{\dagger}v_i)x_1 ) + \mathrm{B}(C^{\dagger}v_i, \omega \mathrm{B}(x_1,C^{\dagger}v_i)x_1) \\
\q(Mv_i) + \mathrm{B}(v_i,Mv_i) &= \q(C^{\dagger}v_i ) + \omega^2\mathrm{B}(x_1,C^{\dagger}v_i)^2\q(x_1) + \omega \mathrm{B}(x_1,C^{\dagger}v_i)^2 \\
\end{align*}

Recall $C^{\dagger} = JC^T$ and we have that  
\[ C^{\dagger}v_i = \sum d_{ij'}x_{j'} + \sum c_{ij'} y_{j'} = d_{i1} x_1, \]
since all other coefficients are zero.  With the computation above we need $Mv_i$ such that
\begin{align*}
\q(Mv_i) + \mathrm{B}(v_i,Mv_i) &= \q(C^{\dagger}v_i ) + \omega^2\mathrm{B}(x_1,C^{\dagger}v_i)^2\q(x_1) + \omega \mathrm{B}(x_1,C^{\dagger}v_i)^2\\
&= \q(d_{i1}x_1 ) + \omega^2\mathrm{B}(x_1,d_{i1}x_1)^2\q(x_1) + \omega \mathrm{B}(x_1,d_{i1}x_1)^2 \\
&= \q(d_{i1}x_1 ) \\
&= \dfrac{d_{i1}^2}{\omega} + \displaystyle\sum_{j=1}^l \left(\dfrac{d_{i1}d_{j1}}{\omega}\right)^2 \q(u_j)
\end{align*}
Setting 
\[ Mv_i = \displaystyle\sum_{j=1}^l \dfrac{d_{i1}d_{j1}}{\omega} u_j, \]
we see that 
\begin{align*}
\q(Mv_i) + \mathrm{B}(v_i,Mv_i) &= \q\left(\displaystyle\sum_{j=1}^l \dfrac{d_{i1}d_{j1}}{\omega} u_j\right) + \mathrm{B}\left(v_i, \displaystyle\sum_{j=1}^l \dfrac{d_{i1}d_{j1}}{\omega} u_j \right) \\
&= \dfrac{d_{i1}^2}{\omega} + \displaystyle\sum_{j=1}^l \left(\dfrac{d_{i1}d_{j1}}{\omega}\right)^2 \q(u_j).
\end{align*}
Notice that this also preserves the bilinear form since
\begin{align*}
\mathrm{B}( T_{x_1,\omega}(v_i), T_{x_1,\omega}(v_{i'}) ) &= \mathrm{B}(v_i + \tau_{x_1,\omega}(C^{\dagger}v_i) + Mv_i,  v_{i'} + \tau_{x_1,\omega}(C^{\dagger}v_{i'}) + Mv_{i'} ) \\
&= \mathrm{B}\left( v_i, \sum \frac{d_{i'1}d_{j1}}{\omega} u_j  \right) \\ 
&+ \mathrm{B}(\tau_{x_1,\omega}(C^{\dagger}v_i), \tau_{x_1,\omega}(C^{\dagger}v_{i'}) ) + \mathrm{B}\left(\sum \frac{d_{i1}d_{j1}}{\omega} u_j, v_{i'} \right) \\
&=\dfrac{d_{i'1}d_{i1}}{\omega} \mathrm{B}(v_i,u_i) + \mathrm{B}(C^{\dagger}v_i, C^{\dagger}v_{i'}) + \dfrac{d_{i1}d_{i'1}}{\omega} \mathrm{B}(u_{i'},v_{i'}) \\
&= \dfrac{d_{i'1}d_{i1}}{\omega} + \mathrm{B}(d_{i1}x_1, d_{i'1}x_1) +  \dfrac{d_{i1}d_{i'1}}{\omega}  \\
&=0.
\end{align*}
\end{proof}

\begin{prop}
Let 
\[ \varphi = \begin{bmatrix}
 \varphi_{\mathcal{U}} & \varphi_{\mathcal{U}} C & \varphi_{\mathcal{U}} A \\
	0 &  \varphi_{\mathcal{X}} & \varphi_{\mathcal{X}} C^{\dagger} \\
	0 & 0 & \varphi_{\mathcal{U}}^*  
	\end{bmatrix} \]
such that $C: \overline{\mathcal{X}} \to \overline{\mathcal{U}}$ for $\varphi_{\mathcal{X}} \in \Sp(\mathrm{B}_{\mathcal{X}},k)$ such that $\q(x+Cx) = \q(\varphi_{\mathcal{X}})$ for all $x \in \overline{\mathcal{X}}$ then there exists an $M: \overline{\mathcal{V}} \to \overline{\mathcal{U}}$ such that 
\[ \q( \varphi(v)) = \q(\varphi_{\mathcal{U}}(Mv) +  \varphi_\mathcal{X}(C^{\dagger}v) + \varphi_{\mathcal{U}}^*(v)). \]
Moreover $\varphi \in \Orth(\q,k)$.
\end{prop}

\begin{proof}
Let us order the basis as before and consider

\begin{equation}
\label{proof_PXC}
\begin{bmatrix}
 \varphi_{\mathcal{U}} & \varphi_{\mathcal{U}} C & \varphi_{\mathcal{U}} A \\
	0 &  \varphi_{\mathcal{X}} & \varphi_{\mathcal{X}} C^{\dagger} \\
	0 & 0 & \varphi_{\mathcal{U}}^*  
	\end{bmatrix}  = \begin{bmatrix}
	 \varphi_{\mathcal{U}} & 0 & \varphi_{\mathcal{U}} (A + M) \\
	0 &  \id & 0\\
	0 & 0 & \varphi_{\mathcal{U}}^*  
	\end{bmatrix}
	\begin{bmatrix}
	\id &  C & M \\
	0 &  \varphi_{\mathcal{X}} & \varphi_{\mathcal{X}} C^{\dagger} \\
	0 & 0 & \id
	\end{bmatrix}
 \end{equation}

	First we choose $v_j$ with $1 \leq j \leq m$ an isotropic basis vector $\q(v_j) = 0$ in $\overline{\mathcal{V}}$.   We choose $Mv_j = (\q( \varphi_\mathcal{X}(C^{\dagger}v_j)  ) + \q(\varphi_{\mathcal{U}}^*(v_j))) u_j$
\begin{align*}
\q(\varphi_{\mathcal{U}}(Mv_j) +  \varphi_\mathcal{X}(C^{\dagger}v_j) + \varphi_{\mathcal{U}}^*(v_j)) &= \q(\varphi_{\mathcal{U}}(Mv_j)) + \q( \varphi_\mathcal{X}(C^{\dagger}v_j)  ) \\
&\ \hspace{1cm} + \q(\varphi_{\mathcal{U}}^*(v_j)) + \mathrm{B}(\varphi_{\mathcal{U}}(Mv_j),\varphi_{\mathcal{U}}^*(v_j))) \\
&= \q(Mv_j) + \q( \varphi_\mathcal{X}(C^{\dagger}v_j)  ) \\
&\ \hspace{2cm} + \q(\varphi_{\mathcal{U}}^*(v_j)) + \mathrm{B}(Mv_j,v_j) = 0
\end{align*}

Now let us choose $v_j$ such that $m < j \leq l$.  Then $\q(v_j)\neq 0$ and $\varphi_{\mathcal{U}}^*(v_j) = v_j$.  We can assume without loss of generality that $\overline{\mathcal{U}}$ is anisotropic, since any isotropic vector from $\overline{\mathcal{U} }$ would not contribute to the overall norm of $\varphi(v_j)$ being both norm zero and having a trivial bilinear form with all anisotropic basis vectors.

Since $\varphi_{\mathcal{X}} \in \Sp(\mathrm{B}_{\mathcal{X}},k)$ and $\mathcal{U}$ is anisotropic we have by Theoreom 2.1.9 in \cite{om78} that $\varphi_{\mathcal{X}}$ can be written as a product of symplectic transvections
\[ \varphi_{\mathcal{X}} = \tau_{x_N,\omega_N} \cdots \tau_{x_2,\omega_2} \tau_{x_1,\omega_1} \]
and by Theorem \ref{uniqueC} $C$ is unique.  So we can factor 
\[ \begin{bmatrix}
	\id &  C & M \\
	0 &  \varphi_{\mathcal{X}} & \varphi_{\mathcal{X}} C^{\dagger} \\
	0 & 0 & \id
	\end{bmatrix} = 
	\begin{bmatrix}
	\id &  C_N & M_N \\
	0 &  \tau_{z_N,\omega_N} & \tau_{z_N,\omega_N} C_N^{\dagger} \\
	0 & 0 & \id
	\end{bmatrix}\cdots 
	\begin{bmatrix}
	\id &  C_1 & M_1 \\
	0 &  \tau_{z_1,\omega_1} & \tau_{z_1,\omega_1} C_1^{\dagger} \\
	0 & 0 & \id
	\end{bmatrix} 
	 \]
where $\varphi_{\mathcal{X}} = \tau_{z_N,\omega_N} \cdots \tau_{z_2,\omega_2}\tau_{z_1,\omega_1}$.  By Lemma \ref{tranv_Sp_M} we have that $M_i$ exists for $1\leq i \leq N$.  Since each $C_i$ is unique, the product on the right hand side gives us the unique $C$ and $M$ on the left hand side.   The bilinear form is preserved by the product since it is preserved by each $T_{x_i,\omega_i}$.

By Lemma \ref{tranv_fix_exist} we can choose $A + M$ such that the first factor on the right hand side of (\ref{proof_PXC}) is an element of $\Orth(\q,k)$.  We have constructed the second factor to be an element of $\Orth(\q,k)$, and so $\varphi \in \Orth(\q,k)$.  

\end{proof}

In order to establish the isomorphism classes of the fixed point groups of involutions on nonsingular vector spaces over fields of characteristic $2$ we need a basis with the vectors organized by specific properties.  Let 
\[ \mathcal{U} = \{u_1,u_2, \ldots, u_l \} \]
 be the set of anisotropic mutually orthogonal vectors in $W$ inducing a diagonal involution $\tau_{\mathcal{U}}$.  We choose a new basis of $\overline{\mathcal{U}}$ indicating the Witt index $m$ of $\overline{\mathcal{U}}$.   We call this basis
\[ \mathcal{U}' =\{ u_1', u_2', \ldots, u_m', u_{m+1}', \ldots, u_l' \}, \]
so that $\overline{\mathcal{U}} = \overline{\mathcal{U}'}$ and $\mathrm{dim}_k(\overline{\mathcal{U}}) = l$ and $\q(u_i')=0$ for $1\leq i \leq m$.  We address the case when $\mathrm{dim}_k(\overline{\mathcal{U}}) = l-1$ at the end of this section.

Now we can choose $\mathcal{V}'$ to be the corresponding hyperbolic pairs making up the nonsingular subspace $\overline{\mathcal{U} \cup \mathcal{V}} = \overline{\mathcal{U}' \cup \mathcal{V}'} \subset W$ such that $\q(v_i') = 0$ for $1\leq i \leq m$ with $\dim(\overline{\mathcal{U} \cup \mathcal{V}}) = 2l$.  We choose a new ordered basis $\mathcal{W}'$ for $W= \overline{\mathcal{W}}$
\[ \mathcal{W} =  \{ u_1,u_2, \ldots, u_l, v_1, v_2, \ldots, v_l , x_1, x_2, \ldots, x_{r-l}, y_1, y_2, \ldots, y_{r-l} \} \]
with $\mathcal{W}' = \mathcal{U}' \cup \mathcal{X} \cup \mathcal{V}'$ with the order below
\[ \mathcal{U}' = \{u_1', u_2', \ldots u_m', u_{m+1}', \ldots, u_l' \},  \]
 \[ \mathcal{X} = \{x_1, x_2, \ldots, x_{r-l}, y_1, y_2, \ldots, y_{r-l} \} \]
 and 
 \[ \mathcal{V}'=\{v_{m+1}', \ldots, v_l', v_1', v_2', \ldots, v_m' \}. \]
 so that the isotropic vectors in $\mathcal{U}'$ and $\mathcal{V}'$ are ordered first and last, respectively.   We can represent an element in $\Orth(\q,k)^{\mathcal{I}_{\tau_{\mathcal{U}}}}$ as the matrix
 \[ P_{\varphi}X_{\varphi}C_{\varphi} = \begin{bmatrix}
	\varphi_{\mathcal{U}} &  \varphi_{\mathcal{U}} C & \varphi_{\mathcal{U}} M \\
	0 &  \varphi_{\mathcal{X}} & \varphi_{\mathcal{X}} C^{\dagger} \\
	0 & 0 & \varphi_{\mathcal{U}}^*
	\end{bmatrix} \]
	 \[ =\begin{bmatrix}
			P_1 &  P_2 &  P_1C_1+P_2C_3  & P_1C_2+P_2C_4 &  P_1M_1+P_2M_3  & P_1M_2+P_2M_4\\
			0 & \id & C_3 & C_4 & M_3 &  M_4  \\
			0 & 0 & X_1 & X_2 & X_1C_4^T + X_2C_3^T &  X_1C_2^T + X_2C_1^T  \\
			0 & 0 & X_3 & X_4 & X_3C_4^T + X_4C_3^T &  X_3C_2^T + X_4C_1^T  \\
			0& 0 & 0 & 0 & \id & P_2^T P_1^* \\
			0& 0 & 0 & 0 & 0 &  P_1^* \\
			\end{bmatrix}
	\]

\begin{align} P_{\varphi} &=\begin{bmatrix}
			P_1 &  P_2 &  0  & 0 &  0  & A_{\varphi} \\
			0 & \id & 0 & 0 & 0 &  0 \\
			0 & 0 & \id &  0 & 0 &  0 \\
			0 & 0 & 0 & \id & 0 &  0  \\
			0& 0 & 0 & 0 & \id & P_2^T P_1^* \\
			0& 0 & 0 & 0 & 0 &  P_1^* \\
			\end{bmatrix} \label{P}  \\			
        X_{\varphi} &= \begin{bmatrix}
			\id  &  0 &  0  & 0 &  0  & 0 \\
			0 & \id & C_3 & C_4 & M_{\varphi} &  0 \\
			0 & 0 & X_1 &  X_2 & X_1C_4^T + X_2C_3^T &  0 \\
			0 & 0 & X_3 & X_4 & X_3C_4^T + X_4C_3^T &  0  \\
			0& 0 & 0 & 0 & \id & 0 \\
			0& 0 & 0 & 0 & 0 &  \id \\
			\end{bmatrix} \label{X} \\
	 C_{\varphi} &= \begin{bmatrix}
			\id  &  0 &  C_1  & C_2 &  M_1  & M_2 \\
			0 & \id & 0 & 0 & 0 &  M_4 \\
			0 & 0 & \id &  0  & 0 &  C_2^T \\
			0 & 0 & 0 & \id  & 0 &  C_1^T  \\
			0& 0 & 0 & 0 & \id & 0 \\
			0& 0 & 0 & 0 & 0 &  \id \\
			\end{bmatrix} \label{C}
\end{align}

We use $\mathcal{U}_{is}$ to denote the isotropic vectors in $\mathcal{U}$ and $\mathcal{U}_{an}$ to denote the anisotropic vectors in $\mathcal{U}$.
			
\begin{prop}
If $\varphi \in \Orth(\q,k)^{\mathcal{I}_{\tau}}$ where $\tau= \tau_{\mathcal{U}}$ then $\varphi = P_{\varphi}X_{\varphi}C_{\varphi}$ from (\ref{P}),(\ref{X}) and (\ref{C}).
\end{prop}

\begin{proof}
The matrix $P_{\varphi}$ is isomorphic to an element in $\Orth(\q_{\mathcal{U}},k)$ with $A_{\varphi}$ unique.  The matrix $X_{\varphi}$ is isomorphic to an element in $\Orth(\q_{\mathcal{U}_{\mathrm{an}} \cup X}, k) \cong H \subset \Sp(\mathrm{B}_{\mathcal{X}},k)$, where $[C_3 | C_4]$ is the unique linear map from Theorem 14.1 in \cite{gr02}.  The maps $C_1, C_2: \overline{X} \to \overline{\mathcal{U}}_{is}$ have isotropic images in $\mathrm{rad}(\overline{\mathcal{U} \cup \mathcal{X}}) = \overline{\mathcal{U}}$ and so have zero norm and zero bilinear form on $\overline{\mathcal{U} \cup \mathcal{X}}$.  Since $P_{\varphi},X_{\varphi} \in \Orth(\q,k)$ and $\varphi \in \Orth(\q,k)^{\mathcal{I}_{\tau}}$ then $C_{\varphi} \in \Orth(\q,k)$ and $P_{\varphi}X_{\varphi}C_{\varphi} = \varphi$ which leaves $\overline{\mathcal{U}}$ invariant.
\end{proof}

We show that each of the matrices $P_{\varphi},X_{\varphi}$ and $C_{\varphi}$ corresponds to a subgroup of $\Orth(\q,k)^{\mathcal{I}_{\tau}}$, and we denote these subgroups by $\mathsf{P}, \mathsf{X}$ and $\mathsf{C}$ respectively.  Next we see what further restrictions are necessary for $\mathsf{P},\mathsf{C}$ and $\mathsf{X}$ to be subgroups of $\Orth(\q,k)$.  

For the following characterizations of these subgroups we assume that the basis for $\overline{\mathcal{U} \cup \mathcal{V}}$ has the following order
\begin{equation}
 \{u_1', u_2', \ldots u_m', u_{m+1}', \ldots, u_l' , v_{m+1}', \ldots, v_l', v_1', v_2', \ldots, v_m'\}
 \label{orderedbasis}
  \end{equation}
where $m$ is the Witt index of $\overline{\mathcal{U} \cup \mathcal{V}}$ and $\q(u_i') = \q(v_i') = 0$ for $1 \leq i \leq m$.

\begin{prop} \label{groupC}
An element $C_{\varphi} \in \mathsf{C} \subset \Orth(\q,k)$ has $C_1, C_2 \in \Mat_{m,s-l}(k)$ and $M_1 \in \Mat_{m,l-m}(k)$, $M_4=M_1^T$  and $M_2 \in \Mat_{m,m}(k)$ such that 
\[ M_2 + M_2^T = C_1C_2^T + C_2C_1^T. \]
\end{prop}
\begin{proof}
Using the ordered basis  (\ref{orderedbasis}) we see that the matrix representation for $\mathrm{B}_{\mathcal{U} \cup \mathcal{V}}$ is of the form
\[ [\mathrm{B}_{  \mathcal{U} \cup \mathcal{V}}] = \begin{bmatrix}
				0 & \id_{l-m} \\
				\id_{m} & 0
				\end{bmatrix}.  \]
Now notice that $M_1: \overline{\mathcal{V}_{\mathrm{is}}} \to \overline{ \mathcal{U}_{\mathrm{an}} }$ and $M_4: \overline{ \mathcal{V}_{\mathrm{an}} } \to \overline{ \mathcal{U}_{\mathrm{is}} }$.  We will compute $\mathrm{B}(Mv_i,v_j)$ for different values of $i$ and $j$.  First, we notice
\begin{align*}
\mathrm{B}(Mv_i,v_j) &= (Mv_i)^T [\mathrm{B}_{  \mathcal{U} \cup \mathcal{V}}] v_j \\
&= v_i^T \begin{bmatrix}
	M_1 & M_2 \\
	0 & M_4
	\end{bmatrix}^T 
	\begin{bmatrix}
		0 & \id_{l-m} \\
		\id_{m} & 0
		\end{bmatrix} v_j \\
&=  v_i^T \begin{bmatrix}
	M_1^T & 0 \\
	M_2^T & M_4^T
	\end{bmatrix}
	\begin{bmatrix}
		0 & \id_{l-m} \\
		\id_{m} & 0
		\end{bmatrix} v_j \\
&= v_i^T
	\begin{bmatrix}
		0 & \id_{l-m} \\
		\id_{m} & 0
		\end{bmatrix}
		 \begin{bmatrix}
	M_4^T & M_2^T \\
	0 & M_1^T
	\end{bmatrix} v_j
\end{align*}
so, we have the following relations
\begin{align*}
\mathrm{B}(M v_i, v_j ) &= \mathrm{B}(M_1 v_i, v_j) \\
&= \mathrm{B}(v_i, M_2^T v_j + M_1^T v_j) \\
&= \mathrm{B}(v_i, M_1^Tv_j) 
\end{align*}
or
\begin{align*}
\mathrm{B}( v_i, M v_j ) &= \mathrm{B}(M_2v_j + M_4v_j, v_i) \\
&= \mathrm{B}(M_4v_j, v_i) \\
&= \mathrm{B}(v_j, M_4^Tv_i) 
\end{align*}
when $1\leq j <m+1\leq i \leq l$
and 
\begin{align*}
\mathrm{B}(M v_i, v_j ) &= \mathrm{B}(M_2 v_i + M_4v_i, v_j) \\
&= \mathrm{B}(M_2 v_i,  v_j) \\
&= \mathrm{B}(v_i, M_2^T v_j + M_1^Tv_j) \\
&= \mathrm{B}(v_i, M_2^T v_j ) 
\end{align*}
when $1\leq i,j \leq m$.

Now we compute $\mathrm{B}(C_{\varphi} v_i, C_{\varphi} v_j)$ for $1\leq i,j \leq m$
\begin{align*}
0 &= \mathrm{B}(C_{\varphi} v_i, C_{\varphi} v_j) \\
&= \mathrm{B}(v_i + (C_1^T + C_2^T + M_4+M_2)v_i, v_j + (C_1^T + C_2^T + M_4+M_2)v_j) \\
&= \mathrm{B}(v_i,M_2v_j) + \mathrm{B}(C_1^Tv_i,C_2^Tv_j) + \mathrm{B}(C_2^Tv_i,C_1^Tv_j) + \mathrm{B}(M_2v_i,v_j) 
\end{align*}
and so, we have
\[ \mathrm{B}(v_i, (M_2 + M_2^T) v_j) = \mathrm{B}(v_i, (C_1C_2^T + C_2C_1^T)v_j) \]
and since $\mathrm{B}$ is nondegenerate on this nonsingular space we have $M_2 + M_2^T = C_1C_2^T + C_2C_1^T$.

On the other hand, if $1 \leq i \leq m < j \leq l$ we have
\begin{align*}
0 &= \mathrm{B}(C_{\varphi} v_i, C_{\varphi} v_j) \\
&= \mathrm{B}(v_i + (C_1^T + C_2^T + M_4+M_2)v_i, v_j + M_1v_j) \\
&= \mathrm{B}(v_i,M_1v_j)+B(M_4v_i,v_j) \\
\mathrm{B}(v_i, M_1 v_j) &= \mathrm{B}(v_i, M_4^T v_j)
\end{align*}
and so, we have $M_1 = M_4^T$.  Further notice that the images of $C_1,C_2$ and $M_1$ are in the totally isotropic subspace of $\overline{\mathcal{U}}$ so have zero norm and zero bilinear form values on $\overline{\mathcal{U} \cup \mathcal{X}}$.

There is a further constraint on the map $M_2$.  To see this we consider the norm of the image under $C_{\varphi}$ on $v_i$ with $1 \leq i \leq m$
\begin{align*}
\q(C_{\varphi}v_i) &= \q(v_i + (C_1^T + C_2^T + M_4+M_2)v_i ) \\
\q(v_i) &= \q(v_i) + \q((C_1^T + C_2^T) v_i) + \q(M_4v_i+ M_2v_i) + \mathrm{B}(v_i, M_4v_i + M_2v_i) \\
\mathrm{B}(v_i,M_2v_i) &=  \q((C_1^T + C_2^T) v_i) + \q(M_1^Tv_i)
\end{align*}
this computation shows us that the diagonal of $M_2$ consists of entries of the form $\q((C_1^T + C_2^T) v_i) + \q(M_1^Tv_i))$.
\end{proof}

\begin{lemma} 
The subgroup $\mathsf{P}$ is isomorphic to $\Orth(\q_{\mathcal{U}}, k)$.
\end{lemma}
\begin{proof}
This the same group from Lemma \ref{tranv_fix_exist} and Proposition \ref{AforP}.
\end{proof}

\begin{lemma}
The subgroup $\mathsf{X}$ is isomorphic to $\Orth(\q_{\mathcal{U}_{\mathrm{an}}' \cup \mathcal{X}},k)$
\end{lemma}
\begin{proof}
This is the group described within \cite{gr02} by Theorem 14.1 and the related discussion.  Theorem 14.1 also appears in this document as Theorem \ref{uniqueC}.
\end{proof}

\begin{lemma} \label{CnormalX}
$\mathsf{X} \mathsf{C} = \mathsf{C} \mathsf{X}  \cong \mathsf{X} \ltimes \mathsf{C}$.
\end{lemma}
\begin{proof}
Notice that $\mathsf{XC}$ is closed under multiplication.  We show that $\mathsf{XC} = \mathsf{CX}$ by verifying that $\mathsf{C}$ is a normal subgroup of $\mathsf{XC}$.  Consider $X_{\varphi} \in \mathsf{X}$ from (\ref{X}) and compute the inverse
\[ X_{\varphi}^{-1} = \begin{bmatrix}
			\id  &  0 &  0  & 0 &  0  & 0 \\
			0 & \id & C_3X_4^T + C_4X_3^T & C_3X_2^T + C_4 X_1^T  & C_3C_4^T + C_4C_3^T + M_{\varphi} &  0 \\
			0 & 0 & X_4^T &  X_2^T & C_4^T &  0 \\
			0 & 0 & X_3^T & X_1^T & C_3^T &  0  \\
			0& 0 & 0 & 0 & \id & 0 \\
			0& 0 & 0 & 0 & 0 &  \id \\
			\end{bmatrix} \]
We define $\varphi_{\mathcal{X}}$ as the symplectic matrix $\bigl[\begin{smallmatrix} X_1 & X_2 \\ X_3 & X_4 \end{smallmatrix} \bigr]$ with the standard symplectic basis for $\overline{\mathcal{X}}$.  Now $X_{\varphi} C_{\varphi} X_{\varphi}^{-1} \in \mathsf{C}$.
\end{proof}

We define the group $\mathsf{C}$ described in Proposition \ref{groupC} to be $\mathsf{A}(\q_{\mathcal{U}},k)$.

\begin{thm}
\label{orth_diag_fixed_pts}
The involution $\tau=\tau_{\mathcal{U}}$ written as a product of transvections induced by vectors $\mathcal{U}$ that are mutually orthogonal has a fixed point group $\Orth(\q,k)^{\mathcal{I}_{\tau}} \cong \Orth(\q_{\mathcal{U}},k) \ltimes ( \Orth(\q_{ \mathcal{U}_{\mathrm{an}} \cup \mathcal{X} }, k) \ltimes \mathsf{A}(\q_{\mathcal{U}},k)$.
\end{thm}
\begin{proof}
We choose a basis for $\overline{\mathcal{U}}$ such that the basis vectors corresponding to the maximal defect in $\overline{\mathcal{U}}$ are listed first and the corresponding defect in $\overline{\mathcal{V}}$, the nonsingular completion of $\overline{\mathcal{U}}$ in $W$, has its basis vectors listed last.  Then by Proposition \ref{nonsing_diag_fixed_pts} $\varphi \in \Orth(\q,k)^{\mathcal{I}_{\tau}}$ has the following form
 \[ \varphi = \begin{bmatrix}
	\varphi_{\mathcal{U}} &  \varphi_{\mathcal{U}} C & \varphi_{\mathcal{U}} M \\
	0 &  \varphi_{\mathcal{X}} & \varphi_{\mathcal{X}} C^{\dagger} \\
	0 & 0 & \varphi_{\mathcal{U}}^*
	\end{bmatrix} \]\small
	
	 \[ =\begin{bmatrix}
			P_1 &  P_2 &  P_1C_1+P_2C_3  & P_1C_2+P_2C_4 &  P_1M_1+P_2M_3  & P_1(M_2+A_{\varphi})+P_2M_4\\
			0 & \id & C_3 & C_4 & M_3 &  M_4  \\
			0 & 0 & X_1 & X_2 & X_1C_4^T + X_2C_3^T &  X_1C_2^T + X_2C_1^T  \\
			0 & 0 & X_3 & X_4 & X_3C_4^T + X_4C_3^T &  X_3C_2^T + X_4C_1^T  \\
			0& 0 & 0 & 0 & \id & P_2^T P_1^* \\
			0& 0 & 0 & 0 & 0 &  P_1^* \\
			\end{bmatrix}
	\] \normalsize
Notice that this factors as $\varphi = P_{\varphi}X_{\varphi}C_{\varphi}$ from (\ref{P}), (\ref{X}) and (\ref{C}) and since $\varphi \in \Orth(\q,k)$ and we have shown that $P_{\varphi}, X_{\varphi} \in \Orth(\q,k)$ so then $C_{\varphi} \in \Orth(\q,k)$.  By Lemma \ref{CnormalX} we have shown the semi-direct product relationship between $\mathsf{X}$ and $\mathsf{C}$ and in a similar way one can verify $\mathsf{XC}$ is a normal subgroup of $\mathsf{P}(\mathsf{XC})$ so we have 
\[ \Orth(\q,k)^{\mathcal{I}_{\tau}} \cong \Orth(\q_{\mathcal{U}},k) \ltimes ( \Orth(\q_{ \mathcal{U}_{\mathrm{an}} \cup \mathcal{X} }, k) \ltimes \mathsf{A}_{\mathcal{U}}(\q,k)). \]
\end{proof}

The only change when the involution is hyperbolic is that the dimension of $\overline{\mathcal{U} \cup \mathcal{V}}$ is $2l-2$ instead of $2l$.  We can see from this fact that there are involutions that are not $\Orth(\q,k)$-conjugate that have the same fixed point group.  This provides other examples of two distinct $G(k)$-conjugacy classes with the same fixed point group over a field of characteristics $2$.

\begin{rmk}
In the case that the inducing set $\mathcal{U}$  for the involution $\tau = \tau_{\mathcal{U}}$ spans an anisotropic vector space $\overline{\mathcal{U}}$, we have $\Orth(\q_{\mathcal{U}},k) = \{ \id \}$ and the matrices $C$ and $M$ are unique up to a given $\varphi_{\mathcal{X}}$.  So we have $\Orth(\q,k)^{\mathcal{I}_{\tau}} \cong \Orth(\q_{\mathcal{U} \cup \mathcal{X}}, k)$.

If $\overline{\mathcal{U}}$ is totally isotropic then $\Orth(\q_{\mathcal{U}},k) \cong \GL(\overline{\mathcal{U}})$ and $\Orth(\q_{\mathcal{U}_{\mathrm{an}} \cup \mathcal{X}},k) \cong \Orth(\q_{\mathcal{X}},k)$.
\end{rmk}

\section{General quadratic spaces}

We now remove the restriction that $W$ is nonsingular.  Let $W_1$ denote a nonsingular subspace of $W$ of maximal dimension, then $W = W_1 \perp \rad(W)$.  A triple $(\rho, Y,\tau)$ will denote an element in $\Orth(\q,k)$ where $\tau$ is a product of symplectic transvections, $\q(\tau(w) + \rho(Yw)) = \q(w)$ and $\rho$ is a product of basic radical involutions.  

\begin{prop}
An involution in $\Orth(\q,k)$  is a map of the form
\[ (\rho, Y, \tau) = \begin{bmatrix}
\rho & \rho Y \\
0 & \tau
\end{bmatrix}, \]
where $\tau = \tau_{\mathcal{U}} \in \Sp(B_{\mathcal{U} \cup \mathcal{V}}, k)$, $\overline{\mathcal{U} \cup \mathcal{V}} \subset W_1$, $Y:W_1 \to \mathrm{rad}(W)$, $\rho \in \Orth(\q_{\mathrm{rad}(W)},k)$ and $\q(\tau(w) + \rho(Yw)) = \q(w)$ with $\tau^2 = \id$, $\rho^2 = id$ and $Y = \rho Y \tau$.
\end{prop}
\begin{proof}
The map $(\rho, Y, \tau) \in \Orth(\q,k)$ must leave $\mathrm{rad}(W)$ invariant.  If you square $(\rho, Y, \tau)$ and set it equal to the identity we need $\rho^2 = \id_{\mathrm{rad}(W)}$ and $\tau^2 = \id$.   In \cite{om78} O'Meara shows us that $\tau$ can be written as a product of symplectic transvections induced by mutually orthogonal vectors in $W_1$.  We call this set of vectors $\mathcal{U} = \{ u_1, \ldots, u_l\}$ and consider a non-singular completion of this totally singular subspace of $W_1$ completed by the vectors $\mathcal{V}=\{ v_1, \ldots, v_l\}$ such that $\mathrm{B}(u_i,v_j) = 1$ if $i=j$ and is zero otherwise.  So $\tau \in \Sp(\mathrm{B}_{\mathcal{U} \cup \mathcal{V}}, k)$.  We have $\rho \in \Orth(\q_{\mathrm{rad}(W)},k)$ since the radical is left invariant by elements in $\Orth(\q,k)$.  The linear transformation $Y : W_1 \to \mathrm{rad}(W)$ can exist since adding on radical vectors leaves the bilinear form invariant.  The quadratic form is left invariant as long as $\q(\tau(w) + \rho(Yw)) = \q(w)$. 
\end{proof}

\begin{corollary}
If $w \in W$ is such that $\tau(w) = w$ then $\rho(Yw) = Yw$.
\end{corollary}

\begin{prop}
For an involution $(\rho, Y, \tau) \in \Orth(\q,k)$ of the form
\[ (\rho, Y, \tau) = \begin{bmatrix}
\rho & \rho Y \\
0 & \tau
\end{bmatrix}, \]
where $\tau = \tau_{\mathcal{U}} \in \Sp(\mathrm{B}_{\mathcal{U} \cup \mathcal{V}}, k)$, $\overline{\mathcal{U} \cup \mathcal{V}} \subset W_1$, $Y:W_1 \to \mathrm{rad}(W)$, $\rho \in \Orth(\q_{\mathrm{rad}(W)},k)$ and $\q(\tau(w) + \rho(Yw)) = \q(w)$ we can choose a basis $\mathcal{U}'$ such that with respect to the new basis $(\rho, Y, \tau)$ is of the form
\[ \begin{bmatrix}
\rho & Y_0\\
0 & \tau_Y
\end{bmatrix} \]
with $\tau_Y \in \Orth(\q,k)$ as a product of orthogonal transvections, $Y_0v_i=0$, and $\q(Y_0 w) = 0$.
\end{prop}
\begin{proof}
Let $\tau = \tau_{a_l,u_l} \cdots \tau_{a_2,u_2}\tau_{a_1,u_1} \in \Sp(\mathrm{B}_{ \mathcal{U} \cup \mathcal{V} }, k)$.  So that $l$ is the length of the transvection we will choose $a_i \neq 0$ for $1 \leq i \leq l$.  Let $\mathcal{X}$ be a basis for $(\overline{ \mathcal{U} \cup \mathcal{V} })^{\perp} \cap W_1$ and notice that if $w \in \overline{ \mathcal{U} \cup \mathcal{X} }$ that $(\rho, Y, \tau)(w) = \tau(w) + \rho(Yw) = w + \rho(Yw)$ and that $\q(Yw) = 0$ since $\mathrm{B}(w,\rho(Yw)) = 0$.

Next for $v_i \in \mathcal{V}$ 
\begin{align*}
\q((\rho, Y, \tau)(v_i) ) &= \q(\tau(v_i) + \rho(Yv_i)) \\
\q(v_i) &= \q\left( v_i + a_i u_i + \rho(Yv_i) \right) \\
\q(v_i) &= \q(v_i) + a_i^2 \q(u_i) + a_i\mathrm{B}(v_i,u_i) + \q(Yv_i) \\
\q(Yv_i) &= a_i^2 \q(u_i) + a_i.
\end{align*}
Now we see that
\begin{align*}
(\rho, Y, \tau)(v_i) &= \tau(v_i) + \rho(Yv_i) \\
&= v_i + a_iu_i + \rho(Yv_i) \\
&= v_i + a_i\mathrm{B}\left(v_i,u_i + \frac{1}{a_i}\rho(Yv_i)\right)\left(u_i+\frac{1}{a_i}\rho(Yv_i)\right) \\
&= \tau_{a_i, u_i+\frac{1}{a_i}\rho(Yv_i)}(v_i) \\
&= \tau_{u_i+\frac{1}{a_i}\rho(Yv_i)}(v_i)
\end{align*}
since
\begin{align*}
\q\left( u_i+\frac{1}{a_i}\rho(Yv_i) \right) &= \q(u_i) + \frac{1}{a_i^2} \q(Yv_i)  \\
&= \q(u_i) + \frac{1}{a_i^2} (a_i^2 \q(u_i) + a_i) \\
&= \frac{1}{a_i}.
\end{align*}
Now we define 
\[ \mathcal{U}_Y = \left\{u_1 + \frac{1}{a_1}\rho(Yv_1), u_2 + \frac{1}{a_2}\rho(Yv_2), \ldots, u_l + \frac{1}{a_l}\rho(Yv_l) \right\}  \]
and consider the maximal nonsingular subspace of $W$, $W_1' = \overline{\mathcal{U}_Y \cup \mathcal{X} \cup \mathcal{V} }$.  Further we define
\[ \tau_Y = \tau_{u_l + \frac{1}{a_l}\rho(Yv_l)} \cdots \tau_{ u_2 + \frac{1}{a_2}\rho(Yv_2)}\tau_{  u_1 + \frac{1}{a_1}\rho(Yv_1) } \]
and $Y_0w = Yw$ for $w \in \overline{ \mathcal{U} \cup \mathcal{X} }$ and $Y_0v = 0$ for $v \in \overline{\mathcal{V}}$.
Using this basis $(\rho, Y, \tau)$ is of the form
\[  \begin{bmatrix}
\rho & Y_0\\
0 & \tau_Y
\end{bmatrix} \]
and we have $\tau_Y \in \Orth(\q,k)$ is a product of orthogonal transvections, $Y_0v_i=0$, and $\q(Y_0 w) = 0$ for all $w\in W_1'$.
\end{proof}

\begin{prop}
The involution $\left[\begin{smallmatrix}\rho & \rho Y \\ 0 & \tau \end{smallmatrix} \right]$ fixes pointwise elements of the form $u_i + \frac{1}{a_i} \rho( Y v_i)$.  Moreover $Y_0 \left(u_i + \frac{1}{a_i} \rho(Yv_i) \right) = 0$.
\end{prop}
\begin{proof}
Let $\tau = \tau_{\mathcal{U}}$ and $u_i \in \mathcal{U}$ with $v_i \in \mathcal{V}$ such that $\mathrm{B}(u_i,v_i) = 1$.  Since $(\rho, Y, \tau)$ is an involution we have that $\rho^2 = \id$, $\tau^2 = \id$ and $Y = \rho Y \tau$.  
First, we note
\[ Yv_i = \rho( Y \tau(v_i)) = \rho (Y (v_i + a_i u_i)). \]
We compute
\begin{align*}
(\rho, Y, \tau)\left(u_i + \frac{1}{a_i} \rho( Y v_i) \right) &= \tau(u_i) + \rho(Yu_i) + \rho\left( \frac{1}{a_i} \rho(Yv_i) \right) \\
&= u_i + \rho(Yu_i) + \frac{1}{a_i} Yv_i \\
&=  u_i + \rho(Yu_i) + \frac{1}{a_i}\rho (Y (v_i + a_i u_i)) \\
&=  u_i + \frac{1}{a_i} \rho( Y v_i).  \\
\end{align*}
Since 
\[ (\rho, Y, \tau) \left( u_i + \frac{1}{a_i} \rho( Y v_i)  \right) = \tau_Y \left(u_i + \frac{1}{a_i} \rho( Y v_i)  \right) \]
we have $Y_0\left(u_i + \frac{1}{a_i} \rho( Y v_i)  \right) = 0$.
\end{proof}

We conclude with a statement about the conjugacy classes and the fixed point groups for involutions of $\Orth(\q,k)$ in the general case for $\q$ along with remarks on special cases for a general quadratic space. 

\begin{prop} \label{conj_gen_inv}
Let $(\rho_1, Y_1, \tau_1) \in \Orth(\q,k)$ be an involution such that $\tau_1 \in \Orth(\q_{W_1'}, k)$ and $\rho_1 \in \Orth(\q_{\mathrm{rad}(W)},k)$, then $(\rho_1,Y_1, \tau_1)$ is $\Orth(\q,k)$-conjugate to $(\rho_2,Y_2, \tau_2) \in \Orth(\q,k)$ if and only if
\begin{enumerate}
\item $\rho_1$ is conjugate to $\rho_2$ by $\psi \in \Orth(\q_{\mathrm{rad}(W)}, k)$
\item $\tau_1$ is conjguate to $\tau_2$ by $\mu  \in \Sp(\mathrm{B}_{W_1'},k)$
\item there exists  $Z:W_1' \to \mathrm{rad}(W)$ such that $\q(w)= \q(\mu(w)) + \q(Zw)$ for all $w \in W_1'$ and 
\[ \psi( \rho_1 Z + Z \tau_1 + Y_1)\mu^{-1} = Y_2. \]
\end{enumerate}
\end{prop}
\begin{proof}
The list of properties in the proposition is equivalent to $\mathcal{I}_{(\psi, Z, \mu)} ( (\rho_1, Y_1, \tau_1)) = (\rho_2, Y_2, \tau_2)$.  To see that we compute 
\[  \begin{bmatrix}
	\psi & \psi Z \\
	0 & \mu
\end{bmatrix}  
\begin{bmatrix}
	\rho_1 & \rho_1 Y_1 \\
	0 & \tau_1
\end{bmatrix}
\begin{bmatrix}
	\psi^{-1} &  Z\mu^{-1} \\
	0 & \mu^{-1}
\end{bmatrix}   = 
\begin{bmatrix}
	\psi \rho_1 \psi^{-1} & \psi( \rho_1 Z + Z\tau_1 + Y_1) \mu^{-1} \\
	0 & \mu \tau_1 \mu^{-1}
\end{bmatrix}.  \]
\end{proof}

\begin{prop}
Let $(\rho, Y, \tau) \in \Orth(\q,k)$ be an involution such that $\tau \in \Orth(\q_{W_1'}, k)$ and $\rho \in \Orth(\q_{\mathrm{rad}(W)},k)$, then $(\psi, D,\mu) \in \Orth(\q,k)$ is fixed  pointwise by $\mathcal{I}_{(\rho,Y, \tau)}$ if and only if
\begin{enumerate}
\item $\psi$ is fixed by conjugation by $\rho$ in $\Orth(\q_{\mathrm{rad}(W)},k)$
\item $\mu$ is fixed by conjugation by $\tau \in \Orth(\q_{W_1'},k)$
\item there exists $Z:W_1' \to \mathrm{rad}(W)$ such that $\q(w)= \q(\mu(w)) + \q(Zw)$ for all $w \in W_1'$ and 
\[ Y + \psi^{-1} Y \mu = Z + \rho Z \tau \]
\end{enumerate}
\end{prop}

\begin{proof}
The list of properties above is equivalent to an element of the form $(\psi, Z, \mu) \in \Orth(\q,k)$ being fixed by  $\mathcal{I}_{(\rho, Y, \tau)}$.  We can compute
\[ \begin{bmatrix}
	\rho & \rho Y \\
	0 & \tau
\end{bmatrix}
\begin{bmatrix}
	\psi & \psi Z \\
	0 & \mu
\end{bmatrix}  
\begin{bmatrix}
	\rho & \rho Y \\
	0 & \tau
\end{bmatrix} = 
\begin{bmatrix}
	\rho \psi \rho & \rho \psi Y \tau + (\rho \psi Z + \rho Y \mu)\tau \\
	0 & \tau \mu \tau
\end{bmatrix}.  \]
So we have $\rho \psi \rho = \psi$ and $\tau \mu \tau = \mu$ right away.   If we consider 
\begin{align*}
\psi Z &= \rho \psi Y \tau + (\rho \psi Z + \rho Y \mu)\tau  \\
&= \rho \psi Y \tau + \rho \psi Z \tau + \rho Y \mu\tau  \\
&= \psi \rho Y \tau + \psi \rho Z \tau + \rho Y  \tau \mu \\
&= \psi Y + \psi \rho Z \tau + Y \mu \\
Z + \rho Z \tau &= Y + \psi^{-1} Y \mu
\end{align*}
we have the result, recalling that in order for $(\rho, Y , \tau)$ to be order $2$ we need $Y= \rho Y \tau$.
\end{proof}

\begin{rmk}
When $\mathrm{rad}(W)$ is anisotropic $\Orth(\q_{\mathrm{rad}(W)},k) = \{\id \}$ and $Z$ is unique to a given $\mu$.

When $\mathrm{rad}(W)$ is totally isotropic we have $\Orth(\q_{\mathrm{rad}(W)},k) \cong \GL_s(k)$, $Z \in \Mat_{s,2r}(k)$ and $\mu \in \Orth(\q_{W_1'},k)$.
\end{rmk}

\bibliographystyle{plain}

\end{document}